\documentclass{amsproc}
\usepackage{amsmath}
\usepackage{amsfonts}
\usepackage{amssymb}
\usepackage{hyperref}
\hypersetup{colorlinks,
citecolor=black,
filecolor=black,
linkcolor=black,
urlcolor=black}
\usepackage{graphics,epstopdf}
\usepackage[pdftex]{graphicx}
\usepackage{newlfont}\newlength{\defbaselineskip}
\usepackage[left=1in,right=1in,top=1in,bottom=0.8in,footskip=0.25in]{geometry}
\usepackage{setspace}
\allowdisplaybreaks
\newtheorem{theorem}{Theorem}[section]
\newtheorem{example}{Example}[section]
\newtheorem{lemma}{Lemma}[section]

\newtheorem{remark}{Remark}[section]

\numberwithin{equation}{section}

\usepackage{cite}

\usepackage{fancyhdr}
\usepackage{graphics}
\usepackage{color,colortbl}
\usepackage{rotating}
\usepackage{fancybox}
\usepackage[table]{xcolor}

\begin{document}
\title{
A Stancu Variant of Sz\'asz-Mirakjan-Kantorovich type Operators: Approximation Properties and Asymptotic Analysis}
\maketitle



\begin{center}
{\bf Rishikesh Yadav$^{1,\star}$,  Ramakanta Meher$^{2,\dag}$}\\
$^{1}$Department of Mathematics and Mathematical Statistics,  Ume\aa~  University, Ume\aa, 
901 87, Sweden\\
$^{2}$Department of Mathematics, Sardar Vallabhbhai National Institute of Technology Surat, Surat-395 007 (Gujarat), India.
\end{center}
\begin{center}
  $^\star$rishikesh.yadav@umu.se,
$^\dag$meher\_ramakanta@yahoo.com
\end{center}

\vskip0.5in
\footnotetext[1]{Corresponding author: rishikesh2506@gmail.com}

\begin{abstract}


This paper investigates the generalized version of the Stancu variant Sz\'asz-Mirakjan Kantorovich type operators. We determine the order of approximation in terms of modulus of continuity and second-order modulus of smoothness; along with this, we derive the result for the rate of convergence in Lipschitz spaces. 
An asymptotic formula is presented to describe the asymptotic behavior of the said operators. Furthermore, some convergence properties and Quantitative Voronovskaya-type as well as Gr\"{u}ss Voronovskaya-type theorems are proved using the weighted modulus of continuity in weighted spaces. In addition, we determine the error bound in the approximation of the functions whose derivatives are of bounded variation. Finally, we provide numerical examples in support of our theoretical findings.

 \end{abstract}
\subjclass \textbf{MSC 2010}: {41A25, 41A35, 41A36}

\textbf{Keywords:}  Sz$\acute{\text{a}}$sz-Mirakjan operators, Kantorovich type operators, modulus of continuity, rate of convergence, weighted approximation, bounded variation

\section{Introduction}\label{sec1}
\emph{Approximation theory} is a crucial part of mathematics, which plays a significant role in analysis, numerical analysis, probability theory, and related areas. \emph{Positive linear operators} are one of the fundamental tools in approximation theory to use the function's approximation along with their approximation properties. 
One of the remarkable works was done by Sz$\acute{\text{a}}$sz and Mirakjan \cite{SO} for approximating the function by the linear positive operators, so-called \emph{Sz$\acute{\text{a}}$sz-Mirakjan operators} in \emph{continuous spaces}. Later, the said operators were generalized by Jain and Pethe \cite{JP}, which are as follows:
\begin{eqnarray}\label{o1}
\mathcal{O}_{n}^{[\alpha]}(g;x)=\sum\limits_{i=0}^{\infty}(1+n\alpha)^{\frac{-x}{\alpha}}\left(\alpha+\frac{1}{n}\right)\frac{x^{(i,-\alpha)}}{i!}g\left(\frac{i}{n}  \right),
\end{eqnarray}
where \begin{eqnarray*}
x^{(i,-\alpha)}=
\begin{cases}
\prod_{j=1}^{i}(x+(j-1)\alpha) & i\geq 1,\\
0 & i=0,
\end{cases}
\end{eqnarray*}
where $g$ is assumed to be an exponential-type function satisfying
 $|g(x)|\leq K e^{Ax},~(x\geq 0,~A>0)$, and
 $K$ is a positive constant, and here  $\alpha=((\alpha_n)_{n\in\mathbb{N}})$ such that $0\leq\alpha_n\leq\frac{1}{n}$.
%
In 2007, Abel and Ivan \cite{AI} studied the local approximation properties of the operators introduced by Jain and Pethe. By introducing the parameter $c=1/(n\alpha)$, they expressed these operators in the equivalent form
\begin{eqnarray}\label{o5}
\mathcal{O}_{n}^{[c]}(g;x)=\left(\frac{c}{c+n} \right)^{nxc} \sum\limits_{i=0}^{\infty} {ncx+i-1\choose i}(1+c)^{-i}g\left(\frac{i}{n} \right),~~~x\geq0,
\end{eqnarray}
where $c={c_n}$, be a real sequence, satisfies $c_n\ge\kappa>0$ for some positive constant $\kappa$. They established local approximation results, obtained estimates for the rate of convergence, and derived the complete asymptotic expansion for this family of operators.
As $c\to\infty$, the operators defined in (\ref{o5}) reduce to classical Sz$\acute{\text{a}}$sz-Mirakjan operators. For $c=1$, the operators in (\ref{o5}), reduce to Lupa\c{s}  operators introduced in \cite{LA1}.  Agratini \cite{AO2} investigated Lupa\c{s} operators, 
where he derived an asymptotic formula and established the rate of convergence of the operators. Furthermore, in the paper \cite{AO1}, the \emph{Kantorovich} version of the Lupa\c{s} operators was introduced along with the approximation properties. These operators are defined by 
\begin{eqnarray*}
K_n(g;x)=n2^{-nx}\sum\limits_{i=0}^\infty\frac{2^{-i}(nx)_i}{i!}\int\limits_{\frac{i}{n}}^{\frac{i+1}{n}}g(t)~dt,~~x\geq0,
\end{eqnarray*} 
 where
 \begin{eqnarray*}
 (nx)_i=
 \begin{cases}
 \prod_{k=1}^{i}(nx+k-1) , & i\geq 1\\
 1                      ,&  i=0.
 \end{cases}
 \end{eqnarray*}
In 2018, Dhamija et al. \cite{DM} introduced a Kantorovich-type generalization of the operators proposed by Jain and Pethe \cite{JP}. For every bounded and integrable function $g$ on $[0,\infty)$, these operators are defined by
\begin{eqnarray}\label{o2}
\mathcal{L}_{n}^{[\alpha]}(g;x)=n\sum\limits_{i=0}^{\infty}(1+n\alpha)^{\frac{-x}{\alpha}}\left(\alpha+\frac{1}{n}\right)^{-i}\frac{x^{(i,-\alpha)}}{i!}\int\limits_{\frac{i}{n}}^{\frac{i+1}{n}}g(t)~dt,
\end{eqnarray}
Motivated by these works, we introduce the operators by incorporating the \emph{Stancu} \cite{DD1} variant of the operators (\ref{o2}) involving two nonnegative parameters $\phi\geq0$, $\psi\geq0$, such that $0\leq\frac{\phi}{\psi}\leq 1$. Let $\beta_n\geq1$ be a strictly increasing sequence of positive real numbers satisfying $(\beta_1=1)$,  $\beta_n\to\infty$ as $n\to\infty$. Then
for every $x\geq 0$, and for any bounded and integrable function defined on $[0,\infty)$, the following operators are introduced:
\begin{eqnarray}\label{o3}
\mathcal{RL}_{n}^{[\alpha]}(g;x)=\beta_n\sum\limits_{i=0}^{\infty}(1+\beta_n\alpha)^{\frac{-x}{\alpha}}\left(\alpha+\frac{1}{\beta_n}\right)^{-i}\frac{x^{(i,-\alpha)}}{i!}\int\limits_{\frac{i}{\beta_n}}^{\frac{i+1}{\beta_n}}g\left(\frac{\beta_n t+\phi}{\beta_n+\psi}\right)~dt,
\end{eqnarray}

The proposed operators include several well-known families as special cases:
\begin{enumerate}
\item{} For $\beta_n=n$, $\phi=\psi=0$,  (\ref{o3}) reduces to the operators (\ref{o2}).
\item{} For $\alpha\to 0$, $\beta_n=n$, $\phi=\psi=0$,   (\ref{o3}) reduces to the classical Sz$\acute{\text{a}}$sz-Mirakjan-Kantorovich operators introduced by Totik \cite{TV1}.
\item{} For $\alpha=\frac{1}{n}$, $\beta_n=n$, $\phi=\psi=0$,  (\ref{o3}) reduces to the operators introduced by Agratini \cite{AO1}.
\item{} For $\beta_n=n$, the proposed family can be viewed as a Stancu-type extension of the Sz$\acute{\text{a}}$sz-Mirakjan-Kantorovich operators.
\end{enumerate}
The main objective of this article is to investigate the approximation properties and to analyze the influence of the parameters $\phi$ and $\psi$ on the approximation process of the proposed operators (\ref{o3}). In particular, we establish several approximation results and study their convergence behaviour. Furthermore, numerical examples and graphical illustrations are presented to validate the theoretical results and demonstrate the effectiveness of the proposed operators.


The remainder of this paper is organized as follows. Section~\ref{sec2} is devoted to the local approximation properties of the proposed operators. In Section~\ref{sec3}, we establish the rate of convergence by proving several approximation theorems. A Voronovskaya-type theorem is presented in Section~\ref{sec4}, while the weighted approximation properties are investigated in Section~\ref{sec5}. Quantitative Voronovskaya-type and Grüss-Voronovskaya-type theorems are established in Section~\ref{sec6}. In Section~\ref{sec7}, we determine the rate of convergence of the operators \eqref{o3} for functions whose derivatives are of bounded variation. Finally, Section~\ref{sec8} presents numerical examples and graphical illustrations comparing the proposed operators, and the concluding section summarizes the main results and discusses possible future directions.


In this section, we establish several auxiliary lemmas that will be used throughout the paper.
\begin{lemma}\label{l2}
For every $x\geq0$ and $n\in\mathbb{N}$, the following equalities hold:
\begin{eqnarray*}
\mathcal{RL}_{n}^{[\alpha]}(1;x)&=&1,\\
\mathcal{RL}_{n}^{[\alpha]}(t;x)&=& \frac{1+2\phi+2x\beta_n}{2(\psi+\beta_n)}\\
\mathcal{RL}_{n}^{[\alpha]}(t^2;x)&=& \frac{1+3 \phi +3 \phi ^2+6 x \beta _n+6 x \phi  \beta _n+3 \alpha x \beta _n^2+3 x^2 \beta _n^2}{3\left(\psi +\beta _n\right){}^2},\\
\mathcal{RL}_{n}^{[\alpha]}(t^3;x)&=& \frac{1}{4(\psi+\beta_n)^3}\Bigg(1+4\phi+6\phi^2+4\phi^3+2x(7+12\phi+6\phi^2)\beta_n+6x(x+\alpha)(3+2\phi)\beta_n^2\\
&&+4x(x^2+3x\alpha+2\alpha^2)\beta_n^3 \Bigg),\\
\mathcal{RL}_{n}^{[\alpha]}(t^4;x)&=&\frac{1}{5 \left(\beta _n+\psi \right){}^4}\Bigg(20 x (\phi +2) \beta _n^3 \left(2 \alpha ^2+x^2+3 \alpha  x\right)+5 x \beta _n^4 \left(6 \alpha ^3+x^3+6 \alpha  x^2+11 \alpha ^2 x\right)\\
&&+15 x \left(2 \phi ^2+6 \phi +5\right) \beta _n^2 (\alpha +x)+10 x \left(2 \phi ^3+6 \phi ^2+7 \phi +3\right) \beta _n+5 \phi ^4+10 \phi ^3+10 \phi ^2+5 \phi +1
 \Bigg),\\
\mathcal{RL}_{n}^{[\alpha]}(t^6;x)&=&\frac{1}{7 \left(\beta _n+\psi \right){}^6}\Bigg( 140 x \beta _n^3\bigg(\phi ^3+6 \phi ^2+13 \phi +10\bigg)  \bigg(2 \alpha ^2+x^2+3 \alpha  x\bigg)\\
&&+35 x \beta _n^4\bigg(3 \phi ^2+15 \phi +20\bigg)  \bigg(6 \alpha ^3+x^3+6 \alpha  x^2+11 \alpha ^2 x\bigg)+42 x (\phi +3) \beta _n^5 \bigg(24 \alpha ^4+x^4\\
&&+10 \alpha  x^3+35 \alpha ^2 x^2+50 \alpha ^3 x\bigg) +7 x \beta _n^6 \bigg(120 \alpha ^5+x^5+15 \alpha  x^4+85 \alpha ^2 x^3+225 \alpha ^3 x^2+274 \alpha ^4 x\bigg)\\
&&+21 x \bigg(5 \phi ^4+30 \phi ^3+75 \phi ^2+90 \phi +43\bigg) \beta _n^2 (\alpha +x)+14 x \bigg(3 \phi ^5+15 \phi ^4+35 \phi ^3+45 \phi ^2+31 \phi +9\bigg) \beta _n\\
&&+7 \phi ^6+21 \phi ^5+35 \phi ^4+35 \phi ^3+21 \phi ^2+7 \phi +1
\Bigg).
\end{eqnarray*}

\end{lemma}
\begin{proof}
By the defined operators, we prove the above equalities
\begin{eqnarray*}
\mathcal{RL}_{n}^{[\alpha]}(1;x)&=& \beta_n\sum\limits_{i=0}^{\infty}(1+\beta_n\alpha)^{\frac{-x}{\alpha}}\left(\alpha+\frac{1}{\beta_n}\right)^{-i}\frac{x^{(i,-\alpha)}}{i!}\frac{1}{\beta_n}\\
&=& (1+\beta_n\alpha)^{\frac{-x}{\alpha}}(1+\beta_n\alpha)^{\frac{x}{\alpha}}=1,\\
\mathcal{RL}_{n}^{[\alpha]}(t;x)&=& \beta_n\sum\limits_{i=0}^{\infty}(1+\beta_n\alpha)^{\frac{-x}{\alpha}}\left(\alpha+\frac{1}{\beta_n}\right)^{-i}\frac{x^{(i,-\alpha)}}{i!}\left(\frac{-x}{\beta_n}+\frac{1}{2\beta_n(\beta_n+\psi)}+\frac{i}{\beta_n(\beta_n+\psi)}+\frac{\phi}{\beta_n(\beta_n+\psi)} \right)\\
&=& -x\left(\frac{1}{1+\beta_n\alpha} \right)^{\frac{-x}{\alpha}}({1+\beta_n\alpha})^{\frac{-x}{\alpha}} +\frac{\left(\frac{1}{1+\beta_n\alpha} \right)^{\frac{-x}{\alpha}}({1+\beta_n\alpha})^{\frac{-x}{\alpha}}}{2(\beta_n+\psi)}+\frac{\beta_nx\left(\frac{1}{1+\beta_n\alpha} \right)^{\frac{-x}{\alpha}}({1+\beta_n\alpha})^{\frac{-x}{\alpha}}}{(\beta_n+\psi)}\\
&&+ \frac{\phi\left(\frac{1}{1+\beta_n\alpha} \right)^{\frac{-x}{\alpha}}({1+\beta_n\alpha})^{\frac{-x}{\alpha}}}{(\beta_n+\psi)}\\
&=& \frac{1+2\phi+2x\beta_n}{2(\psi+\beta_n)}.
\end{eqnarray*}
\end{proof}

Similarly, we can prove other equalities.

\begin{lemma}\label{lm3}
Let $\Phi_{n,m}^{\alpha}(x)=\mathcal{RL}_{n}^{[\alpha]}((t-x)^m;x),~~m=1, 2, 3, 4, 6$  where $t,x\in[0,\infty)$. Then for all $n\in\mathbb{N}$, we obtain
\begin{eqnarray*}
\Phi_{n,1}^{\alpha}(x)&=& \frac{1+2\phi-2x\psi}{2(\psi+\beta_n)},\\
\Phi_{n,2}^{\alpha}(x)&=& \frac{1+3\phi+3\phi^2-3x\psi-6x\phi\psi+3x^2\psi^2+3x\beta_n+3x\alpha\beta_n^2}{3(\psi+\beta_n)^2},\\
\Phi_{n,3}^{\alpha}(x)&=& \frac{1}{4(\psi+\beta_n)^3}\Bigg(1+4\phi^3-4x\psi+6x^2\psi^2-4x^3\psi^3+\phi^2(6-12x\psi)+4\phi(1-3x\psi+3x^2\psi^2)\\
&&+2x(5+6\phi-6x\psi)\beta_n+6x\alpha(3+2\phi+2x\psi)\beta_n^2+8x\alpha^2\beta_n^3 \Bigg),\\
\Phi_{n,4}^{\alpha}(x)&=& \frac{1}{5 \left(\beta _n+\psi \right){}^4}\Bigg(15 x \beta _n^2 \left(\alpha  \left(2 \phi ^2+6 \phi +5\right)+2 \alpha  x^2 \psi ^2+x (1-2 \alpha  \psi  (2 \phi +3))\right)\\
&&+5 x \beta _n \left(6 x^2 \psi ^2-10 x \psi -2 \phi  (6 x \psi -5)+6 \phi ^2+5\right)\\
&&+15 \alpha ^2 x \beta _n^4 (2 \alpha +x)+10 \alpha  x \beta _n^3 (4 \alpha  (\phi +2)+x (3-4 \alpha  \psi ))\\
&&+5 x^4 \psi ^4-10 x^3 \psi ^3+10 x^2 \psi ^2+10 \phi ^2 \left(3 x^2 \psi ^2-3 x \psi +1\right)\\
&&-5 \phi  \left(4 x^3 \psi ^3-6 x^2 \psi ^2+4 x \psi -1\right)-5 x \psi +\phi ^3 (10-20 x \psi )+5 \phi ^4+1
\Bigg),\\
\Phi_{n,6}^{\alpha}(x)&=&\frac{1}{7 \left(\beta _n+\psi \right){}^6}\Bigg(  -21 \alpha ^2 x \beta _n^5 \left(-48 \alpha ^2 (\phi +3)+5 x^2 (8 \alpha  \psi -3)+2 \alpha  x (24 \alpha  \psi -20 \phi -75)\right)\\
&&+35 \alpha ^3 x \beta _n^6 \left(24 \alpha ^2+3 x^2+26 \alpha  x\right)+35 \alpha  x \beta _n^4 \bigg(6 \alpha ^2 \bigg(3 \phi ^2+15 \phi +20\bigg)+9 \alpha  \psi ^2 x^3\\
&&+3 x^2 \bigg(6 \alpha ^2 \psi ^2-\alpha  \psi  (6 \phi +23)+3\bigg)+\alpha  x \bigg(-90 \alpha  \psi +\phi  (69-36 \alpha  \psi )+9 \phi ^2+116\bigg)\bigg)\\
&&-35 x \beta _n^3 \bigg(-8 \alpha ^2 \bigg(\phi ^3+6 \phi ^2+13 \phi +10\bigg)+2 \alpha  \psi ^2 x^3 (4 \alpha  \psi -9)-3 x^2 \bigg(8 \alpha ^2 \psi ^2 (\phi +2)\\
&&-2 \alpha  \psi  (6 \phi +11)+1\bigg)+2 \alpha  x \bigg(52 \alpha  \psi +3 \phi ^2 (4 \alpha  \psi -3)+\phi  (48 \alpha  \psi -33)-33\bigg)\bigg)\\
&& +7 x \beta _n^2 \bigg(3 \alpha  \bigg(5 \phi ^4+30 \phi ^3+75 \phi ^2+90 \phi +43\bigg)+15 \alpha  x^4 \psi ^4-15 x^3 \psi ^2 (2 \alpha  \psi  (2 \phi +3)-3)\\
&&+15 x^2 \psi  \bigg(15 \alpha  \psi +6 \alpha  \psi  \phi ^2+6 \phi  (3 \alpha  \psi -1)-7\bigg)-5 x \bigg(54 \alpha  \psi +12 \alpha  \psi  \phi ^3+9 \phi ^2 (6 \alpha  \psi -1)\\
&&+3 \phi  (30 \alpha  \psi -7)-14\bigg)\bigg) +7 x \beta _n \bigg(15 x^4 \psi ^4-50 x^3 \psi ^3+75 x^2 \psi ^2+15 \phi ^2 \bigg(6 x^2 \psi ^2-10 x \psi +5\bigg)\\
&&-2 \phi  \bigg(30 x^3 \psi ^3-75 x^2 \psi ^2+75 x \psi -28\bigg)-56 x \psi +\phi ^3 (50-60 x \psi )+15 \phi ^4+17\bigg)\\
&& +7 x^6 \psi ^6-21 x^5 \psi ^5+35 x^4 \psi ^4-35 x^3 \psi ^3+21 x^2 \psi ^2+35 \phi ^4 \bigg(3 x^2 \psi ^2-3 x \psi +1\bigg)\\
&&-35 \phi ^3 \bigg(4 x^3 \psi ^3-6 x^2 \psi ^2+4 x \psi -1\bigg) +21 \phi ^2 \bigg(5 x^4 \psi ^4-10 x^3 \psi ^3+10 x^2 \psi ^2-5 x \psi +1\bigg)\\
&&-7 \phi  \bigg(6 x^5 \psi ^5-15 x^4 \psi ^4+20 x^3 \psi ^3-15 x^2 \psi ^2+6 x \psi -1\bigg)-7 x \psi +\phi ^5 (21-42 x \psi )+7 \phi ^6+1
\Bigg).
\end{eqnarray*}

\end{lemma}
\begin{proof}
Using the above Lemma \ref{l2}, we  have
\begin{eqnarray*}
\Phi_{n,1}^{\alpha}(x)&=& \mathcal{RL}_{n}^{[\alpha]}(t;x)-x\\
&=& \frac{1+2\phi-2x\psi}{2(\psi+\beta_n)}.\\
\Phi_{n,2}^{\alpha}(x)&=& \mathcal{RL}_{n}^{[\alpha]}(t^2;x)-2x\mathcal{RL}_{n}^{[\alpha]}(t;x)+x^2\\
&=& \frac{1+3 \phi +3 \phi ^2+6 x \beta _n+6 x \phi  \beta _n+3 \alpha x \beta _n^2+3 x^2 \beta _n^2}{3\left(\psi +\beta _n\right){}^2}-2x\frac{1+2\phi+2x\beta_n}{2(\psi+\beta_n)}+x^2\\
&=& \frac{1+3\phi+3\phi^2-3x\psi-6x\phi\psi+3x^2\psi^2+3x\beta_n+3x\alpha\beta_n^2}{3(\psi+\beta_n)^2}.\\
\Phi_{n,3}^{\alpha}(x)&=& \mathcal{RL}_{n}^{[\alpha]}(t^3;x)-x^3+3x^2\mathcal{RL}_{n}^{[\alpha]}(t;x)-3x\mathcal{RL}_{n}^{[\alpha]}(t^2;x)\\
&=& \frac{1}{4(\psi+\beta_n)^3}\Bigg(1+4\phi+6\phi^2+4\phi^3+2x(7+12\phi+6\phi^2)\beta_n+6x(x+\alpha)(3+2\phi)\beta_n^2\\
&&+4x(x^2+3x\alpha+2\alpha^2)\beta_n^3 \Bigg)-x^3+3x^2\left(\frac{1+2\phi+2x\beta_n}{2(\psi+\beta_n)} \right)\\
&&-3x\left(\frac{1+3\phi+3\phi^2-3x\psi-6x\phi\psi+3x^2\psi^2+3x\beta_n+3x\alpha\beta_n^2}{3(\psi+\beta_n)^2}\right).
\end{eqnarray*}
Hence, the proof is completed.
\end{proof}
\begin{remark}
The operators (\ref{o3}) can be written as:
\begin{eqnarray}\label{co3}
\mathcal{RL}_{n}^{[\alpha]}(g;x)=(\beta_n+\psi)\sum\limits_{i=0}^{\infty}(1+\beta_n\alpha)^{\frac{-x}{\alpha}}\left(\alpha+\frac{1}{\beta_n}\right)^{-i}\frac{x^{(i,-\alpha)}}{i!}\int\limits_{\frac{i+\phi}{\beta_n+\psi}}^{\frac{1+i+\phi}{\beta_n+\psi}}g(s)~ds.
\end{eqnarray}
\end{remark}

\begin{remark}
One can write the above operators into integral representation as:
\begin{eqnarray}\label{o4}
\mathcal{RL}_{n}^{[\alpha]}(g;x)=\int\limits_0^\infty \mathcal{Y}_n^{[\alpha]}(x,s)g(s)~ds,
\end{eqnarray}
where $\mathcal{Y}_n^{[\alpha]}(x,s)=(\beta_n+\psi)\sum\limits_{i=0}^\infty l_{n,i}^{[\alpha]}(x)\chi_{n,i}(x,s)$, where $\chi_{n,i}(x,s)$ is the characteristic function of the interval $[\frac{i+\phi}{\beta_n+\psi}, \frac{1+i+\phi}{\beta_n+\psi}]$ with respect to $[0,\infty)$ and $l_{n,i}^{[\alpha]}(x)=(1+\beta_n\alpha)^{\frac{-x}{\alpha}}\left(\alpha+\frac{1}{\beta_n}\right)^{-i}\frac{x^{(i,-\alpha)}}{i!}$. 
\end{remark}

\begin{remark}\label{re1}
For each $g\in C_B[0,\infty)$, where $C_B[0,\infty)$ be the space of all continuous and bounded functions defined on $[0,\infty)$,  we have
$$|\mathcal{RL}_{n}^{[\alpha]}(g;x)|\leq\|g\|_{C_B[0,\infty)},~~x\in[0,\infty).$$

\end{remark}

\begin{lemma}\label{l1}
For every $x\geq 0$ and $\max\alpha=\frac{1}{\beta_n}$, then it holds:
\begin{eqnarray*}
\underset{\beta_n\to\infty}\lim\{\beta_n\Phi_{n,1}^{\alpha}(x)\}&=& \frac{1}{2}-x \psi +\phi,\\
\underset{\beta_n\to\infty}\lim\{\beta_n\Phi_{n,2}^{\alpha}(x)\}&=& 2x,\\
\underset{\beta_n\to\infty}\lim\{\beta_n^2\Phi_{n,4}^{\alpha}(x)\}&=& 12x^2,\\
\underset{\beta_n\to\infty}\lim\{\beta_n^3\Phi_{n,6}^{\alpha}(x)\}&=& 120x^3.
\end{eqnarray*}

\end{lemma}

\section{Local Approximation}\label{sec2}
This section is devoted to the local approximation properties of the proposed operators (\ref{o3}). 
Let $C_B[0,\infty)$ be the space of continuous and bounded functions defined on $[0,\infty)$ equipped with 
supremum norm $\|g \|_{C_B[0,\infty)}=\underset{x\geq 0}\sup\{|g(x)|\}$. We determine the rate of convergence using the Steklov mean function \cite{AO1}. 
The Steklov function associated with the function $g\in C_B[0,\infty)$ is defined by:
\begin{eqnarray*}
g_h(x)=\frac{4}{h^2}\int\limits_0^{\frac{h}{2}}\int\limits_0^{\frac{h}{2}}(2g(x+u+v)-g(x+2(u+v)))du~dv.
\end{eqnarray*}
The following properties are easily verified:
\begin{enumerate}\label{st}
\item{} $\| g_h-g\|_{C_B[0,\infty)}\leq \omega_2(g,h)$,
\item{} $\|g_h'\|_{C_B[0,\infty)}\leq \frac{5}{h}\omega(g,h)~\text{and}~\|g_h''\|_{C_B[0,\infty)}\leq \frac{9}{h^2}\omega_2(g,h),~~\text{for}~g_h',g_h''\in C_B[0,\infty),$
\end{enumerate}
where $\omega(g,\delta)=\underset{x,u,v\geq 0}\sup\underset{|u-v|\leq \delta}\sup|g(x+u)-g(x+v)|$ and $\omega_2(g,\delta)=\underset{x,u,v\geq 0}\sup\underset{|u-v|\leq \delta}\sup|g(x+2u)-2g(x+u+v)|+g(x+2v)$ are the usual modulus of continuity and the second-order modulus of continuity, respectively.
\begin{theorem}
Let $g\in C_B[0,\infty)$. Then for every $x\in[0,\infty)$, the following inequality holds:
\begin{eqnarray*}
|\mathcal{RL}_{n}^{[\alpha]}(g;x)-g(x)|\leq 5\left(\omega\left(g,\sqrt{\Phi_{n,2}^{\alpha}(x)} \right)+\frac{9}{10}\omega_2\left(g,\sqrt{\Phi_{n,2}^{\alpha}(x)} \right)\right).
\end{eqnarray*}
\end{theorem}

\begin{proof}
For every $x\geq 0$. Using the Steklov mean, we obtain:
\begin{eqnarray*}
|\mathcal{RL}_{n}^{[\alpha]}(g;x)-g(x)|\leq \mathcal{RL}_{n}^{[\alpha]}(|g-g_h|;x)+ |\mathcal{RL}_{n}^{[\alpha]}(g_h-g_h(x);x)|+|g_h-g(x)|.
\end{eqnarray*} 
Since $g\in C_B[0,\infty)$, it holds $|\mathcal{RL}_{n}^{[\alpha]}(g;x)|\leq\|g\|_{C_B[0,\infty)},~~x\in[0,\infty)$. Using the Steklov mean property (\ref{st}), we obtain
\begin{eqnarray}
\mathcal{RL}_{n}^{[\alpha]}(|g-g_h|;x)\leq \|\mathcal{RL}_{n}^{[\alpha]}(g-g_h)\|_{C_B[0,\infty)}\leq \|g-g_h\|_{C_B[0,\infty)}\leq  \omega_2(g,\delta).
\end{eqnarray}
Applying Taylor's expansion together with the Cauchy–Schwarz inequality,
\begin{eqnarray*}
|\mathcal{RL}_{n}^{[\alpha]}(g_h-g_h(x);x)|\leq \|g_h'\|_{C_B[0,\infty)} \sqrt{\mathcal{RL}_{n}^{[\alpha]}((t-x)^2;x)}+\frac{\|g_h''\|_{C_B[0,\infty)}}{2!}\mathcal{RL}_{n}^{[\alpha]}((t-x)^2;x).
\end{eqnarray*}
Using the second property of Steklov mean (\ref{st}), we have
\begin{eqnarray*}
|\mathcal{RL}_{n}^{[\alpha]}(g_h-g_h(x);x)|\leq \frac{5}{h}\omega(g,\delta) \sqrt{\Phi_{n,2}^{\alpha}(x)}+\frac{9}{2h^2}\omega_2(g,\delta)\Phi_{n,2}^{\alpha}(x).
\end{eqnarray*}
Choosing $h=\sqrt{\Phi_{n,2}^{\alpha}(x)}$ yields the desired estimate. 
\end{proof}
Let $C_B^k[0,\infty)=\{g\in C_B[0,\infty)|g',g'',\cdots, g^k\in C_B[0,\infty)\}$. We next derive an upper bound for the approximation error using the usual modulus of continuity.
\begin{theorem}
If $g\in C_B^1[0,\infty)$, then for every $x\in [0,\infty)$,  an inequality holds:
\begin{eqnarray*}
|\mathcal{RL}_{n}^{[\alpha]}(g;x)-g(x)|\leq |g'(x)|\left|\frac{1+2\phi+2x\psi}{2(\psi+\beta_n)} \right|+ 2\sqrt{\Phi_{n,2}^{\alpha}(x)}\omega\left(g',\sqrt{\Phi_{n,2}^{\alpha}(x)} \right).
\end{eqnarray*}
\end{theorem}
\begin{proof}
For $g\in C_B^1[0,\infty)$ and $t,x\in [0,\infty)$, Taylor's integral remainder formula gives
\begin{eqnarray*}
g(t)-g(x)=g'(x)(t-x)+\int\limits_x^t(g'(u)-g'(x))~du.
\end{eqnarray*}
Applying the operators $\mathcal{RL}_{n}^{[\alpha]}(.;x)$ to both sides, we get
\begin{eqnarray*}
|\mathcal{RL}_{n}^{[\alpha]}(g;x)-g(x)|&\leq & |g'(x)|\left|\mathcal{RL}_{n}^{[\alpha]}((t-x);x)|+\mathcal{RL}_{n}^{[\alpha]}\left(\int\limits_x^t(g'(u)-g'(x))~du \right)\right|.
\end{eqnarray*} 
Using the following standard property of the modulus of continuity,
\begin{eqnarray}
|g(t)-g(x)|\leq \omega(g,\delta)\left(\frac{|t-x|}{\delta}+1 \right).
\end{eqnarray}
Therefore, 
\begin{eqnarray}
\left|\int\limits_x^t(g'(u)-g'(x))~du \right|\leq \omega(g',\delta) \left(\frac{|t-x|}{\delta}+1 \right)|t-x|.
\end{eqnarray} 
Therefore,
\begin{eqnarray*}
|\mathcal{RL}_{n}^{[\alpha]}(g;x)-g(x)|&\leq & |g'(x)||\mathcal{RL}_{n}^{[\alpha]}((t-x);x)|+ \omega(g',\delta) \left(\frac{\mathcal{RL}_{n}^{[\alpha]}((t-x)^2;x)}{\delta}+\mathcal{RL}_{n}^{[\alpha]}(|t-x|;x) \right).
\end{eqnarray*}
Applying the  Cauchy-Schwarz inequality, we obtain 
\begin{eqnarray*}
|\mathcal{RL}_{n}^{[\alpha]}(g;x)-g(x)|&\leq & |g'(x)|\left|\frac{1+2\phi+2x\psi}{2(\psi+\beta_n)} \right|+\omega(g',\delta)\left(\frac{\sqrt{\Phi_{n,2}^{\alpha}(x)}}{\delta}+1 \right)\sqrt{\Phi_{n,2}^{\alpha}(x)}.
\end{eqnarray*}
By choosing, $\delta=\sqrt{\Phi_{n,2}^{\alpha}(x)}$, we get our required result.
\end{proof}

In 1988, Lenze \cite{LB} introduced a Lipschitz maximal function of order $j\in (0,1]$, which is defined below: 
\begin{eqnarray}\label{lm}
{\varpi}_j(g,x)=\underset{t\neq x}\sup\frac{|g(t)-g(x)|}{|t-x|^j},~~~t\in[0,\infty).
\end{eqnarray}  
Next, we estimate the upper bound in the approximation by the proposed operators with the function in terms of the Lipschitz maximal function. This, shows the convergence rate of the operators.
\begin{theorem}
Let $g\in C_B[0,\infty)$, $x\geq 0$. 
Then the following estimate holds:
\begin{eqnarray*}
|\mathcal{RL}_{n}^{[\alpha]}(g;x)-g(x)|&\leq & {\varpi}_j(g,x)\left(\Phi_{n,2}^{\alpha}(x) \right)^{\frac{j}{2}},~~0<j\leq 1.  
\end{eqnarray*} 
\end{theorem}
\begin{proof}
By the definition of ${\varpi}_j(g,x)$, and then applying positivity and linearity for the operators (\ref{o3}), we have
\begin{eqnarray*}
|\mathcal{RL}_{n}^{[\alpha]}(g;x)-g(x)|\leq {\varpi}_j(g,x)\mathcal{RL}_{n}^{[\alpha]}(|t-x|^j;x).
\end{eqnarray*}
Using H$\ddot{\text{o}}lder$ inequality with $p=\frac{2}{j},~q=\frac{2}{2-j}$ along with  the Lemma \ref{lm3}, we obtain
\begin{eqnarray*}
|\mathcal{RL}_{n}^{[\alpha]}(g;x)-g(x)|&\leq & {\varpi}_j(g,x)\left(\Phi_{n,2}^{\alpha}(x) \right)^{\frac{j}{2}}.
\end{eqnarray*}
Hence, the proof is completed.
\end{proof}
Our next theorem is based on a Lipschitz-type space with two parameters, proposed in \cite{MAH1}. Let $\nu_1\geq 0, \nu_2\geq 0$ be fixed numbers. Then the space is defined as follows:
\begin{eqnarray*}
Lip_M^{\nu_1,\nu_2}(j)=\left\{g\in C_B[0,\infty):|g(y)-g(x)|\leq M\frac{|y-x|^j}{(y+\nu_1x^2+\nu_2x)^{\frac{j}{2}}},~~x,y\in(0,\infty) \right\},
\end{eqnarray*}
where $M$ is a positive constant and $j\in(0,1]$. Now, we are ready to obtain the approximate result.
\begin{theorem}
Let $g\in Lip_M^{\nu_1,\nu_2}(j)$ and $j\in (0,1]$. Then for all $x\in[0,\infty)$, an upper bound in the approximation of the function is given by:
\begin{eqnarray*}
|\mathcal{RL}_{n}^{[\alpha]}(g;x)-g(x)|\leq M \left(\frac{\Phi_{n,2}^{\alpha}(x)}{x(\nu_1x+\nu_2)} \right)^{\frac{j}{2}}, 
\end{eqnarray*}
where $\Phi_{n,2}^{\alpha}(x)=\mathcal{RL}_{n}^{[\alpha]}((t-x)^2;x)$.
\end{theorem}

\begin{proof}
We consider the cases $j=1$ and $0<j<1$ separately. We know that the proposed operators are positive, linear, and reproduce the constant term.

\textbf{Case 1.} When $j=1$, then we can write 
\begin{eqnarray*}
|\mathcal{RL}_{n}^{[\alpha]}(g;x)-g(x)|&\leq & \mathcal{RL}_{n}^{[\alpha]}(|g(t)-g(x)|;x)\\
&\leq & (\beta_n+\psi)\sum\limits_{i=0}^{\infty}(1+\beta_n\alpha)^{\frac{-x}{\alpha}}\left(\alpha+\frac{1}{\beta_n}\right)^{-i}\frac{x^{(i,-\alpha)}}{i!}\int\limits_{\frac{i+\phi}{\beta_n+\psi}}^{\frac{1+i+\phi}{\beta_n+\psi}}|g(t)-g(x)|~dt\\
&\leq & M (\beta_n+\psi)\sum\limits_{i=0}^{\infty}(1+\beta_n\alpha)^{\frac{-x}{\alpha}}\left(\alpha+\frac{1}{\beta_n}\right)^{-i}\frac{x^{(i,-\alpha)}}{i!}\int\limits_{\frac{i+\phi}{\beta_n+\psi}}^{\frac{1+i+\phi}{\beta_n+\psi}} \frac{|t-x|}{(t+\nu_1x^2+\nu_2x)^{\frac{1}{2}}}~dt\\
&\leq & M (\beta_n+\psi)\sum\limits_{i=0}^{\infty}(1+\beta_n\alpha)^{\frac{-x}{\alpha}}\left(\alpha+\frac{1}{\beta_n}\right)^{-i}\frac{x^{(i,-\alpha)}}{i!}\int\limits_{\frac{i+\phi}{\beta_n+\psi}}^{\frac{1+i+\phi}{\beta_n+\psi}} \frac{|t-x|}{(\nu_1x^2+\nu_2x)^{\frac{1}{2}}}~dt\\
&=& \frac{M}{(\nu_1x^2+\nu_2x)^{\frac{1}{2}}}\mathcal{RL}_{n}^{[\alpha]}(|t-x|;x)\\
&\leq & M \left(\frac{\Phi_{n,2}^{\alpha}(x)}{x(\nu_1x+\nu_2)} \right)^{\frac{1}{2}}.
\end{eqnarray*}
\textbf{Case 2.} Now let $j\in (0,1)$. Let $p=\frac{1}{j},~q=\frac{1}{1-j}$, using the linear and positivity properties of the operators, and then using H$\ddot{\text{o}}$lder's inequality, we have
\begin{eqnarray*}
|\mathcal{RL}_{n}^{[\alpha]}(g;x)-g(x)|&\leq & (\beta_n+\psi)\sum\limits_{i=0}^{\infty}(1+\beta_n\alpha)^{\frac{-x}{\alpha}}\left(\alpha+\frac{1}{\beta_n}\right)^{-i}\frac{x^{(i,-\alpha)}}{i!}\int\limits_{\frac{i+\phi}{\beta_n+\psi}}^{\frac{1+i+\phi}{\beta_n+\psi}}|g(t)-g(x)|~dt\\
&\leq & \sum\limits_{i=0}^{\infty}(1+\beta_n\alpha)^{\frac{-x}{\alpha}}\left(\alpha+\frac{1}{\beta_n}\right)^{-i}\frac{x^{(i,-\alpha)}}{i!}\left((\beta_n+\psi)\int\limits_{\frac{i+\phi}{\beta_n+\psi}}^{\frac{1+i+\phi}{\beta_n+\psi}}|g(t)-g(x)|^{\frac{1}{j}}~dt\right)^{j}\\
&\leq & M \left( \sum\limits_{i=0}^{\infty}(1+\beta_n\alpha)^{\frac{-x}{\alpha}}\left(\alpha+\frac{1}{\beta_n}\right)^{-i}\frac{x^{(i,-\alpha)}}{i!}(\beta_n+\psi)\int\limits_{\frac{i+\phi}{\beta_n+\psi}}^{\frac{1+i+\phi}{\beta_n+\psi}}\frac{|t-x|}{(t+\nu_1x^2+\nu_2x)^{\frac{1}{2}}}~dt\right)^{j}\\
&\leq & \frac{M}{(\nu_1x^2+\nu_2x)^{\frac{j}{2}}} \left(\mathcal{RL}_{n}^{[\alpha]}(|t-x|;x)\right)^{j}\\
&\leq & M \left(\frac{\Phi_{n,2}^{\alpha}(x)}{x(\nu_1x+\nu_2)} \right)^{\frac{j}{2}}.
\end{eqnarray*}
Hence, the proof is completed.
\end{proof}

\section{Direct Result}\label{sec3}
This section is devoted to a direct approximation result for the proposed operators. To establish the direct approximation theorem, we employ Peetre's $K$-functional, defined for $\delta>0$, by:
\begin{eqnarray*}
K_2(g;\delta)=\underset{g_1\in C_B^2[0,\infty)}\inf\{\|g-g_1\|+\delta \|g_1''\| \},
\end{eqnarray*} 
where $C_B^2[0,\infty)=\{g\in C_B[0,\infty):g', g''\in C_B[0,\infty)\}$.\\

The following well-known relation between Peetre's $K$-functional and the second-order modulus of smoothness is given in \cite{q1}:
\begin{eqnarray}\label{pe1}
K_2(g;\delta)\leq M \omega_2(g;\sqrt{\delta}),
\end{eqnarray} 
where the second-order modulus of smoothness is defined by 
\begin{eqnarray*}
\omega_2(g;\sqrt{\delta})=\underset{h\in[0,\sqrt{\delta}],x\in[0,\infty)}\sup \{|g(x+2h)-2g(x+h)+g(x)|\}, ~g\in C_B[0,\infty).
\end{eqnarray*}
The first-order modulus of continuity is defined by
$$\omega(g;\sqrt{\delta})=\underset{h\in[0,\sqrt{\delta}],x\in[0,\infty)}\sup \{|g(x+h)-g(x)|\}, ~g\in C_B[0,\infty).$$
The next theorem addresses the idea of the convergence rate in terms of the modulus of continuity, which gives a convergence guarantee.
\begin{theorem}
Let $g\in C_B[0,\infty)$. Then there exists a positive constant $M_1$, such that
\begin{eqnarray*}
|\mathcal{RL}_{n}^{[\alpha]}(g;x)-g(x)|\leq M_1\omega_2(g;\sqrt{\gamma_n(x)})+\omega(g;\eta_n(x)),~~x\in[0,\infty),
\end{eqnarray*}
where 
\begin{eqnarray*}
\gamma_n(x)=\Phi_{n,2}^{\alpha}(x)+\left(\frac{1+2\phi+2x\beta_n}{2(\psi+\beta_n)}-x \right)^2,~~\eta_n(x)=\left|\frac{1+2\phi-2x\psi}{2(\psi+\beta_n)} \right|,~\text{and}~~\Phi_{n,2}^{\alpha}(x)=\mathcal{RL}_{n}^{[\alpha]}((t-x)^2;x).
\end{eqnarray*}
\end{theorem}
\begin{proof}
For $x\geq 0$, consider the auxiliary operators  $\{\O_{n}^{[\alpha]}(g;x)\}$, are defined by
\begin{eqnarray*}
\O_{n}^{[\alpha]}(g;x)=\mathcal{RL}_{n}^{[\alpha]}(g;x)+g(x)-g\left(\frac{1+2\phi+2x\beta_n}{2(\psi+\beta_n)} \right),
\end{eqnarray*}
and by Lemma \ref{lm3}, one can get $\O_{n}^{[\alpha]}((t-x);x)=0$. By Taylor's integral remainder formula and let $f\in C_B^2[0,\infty)$, we have
\begin{eqnarray*}
g(t)-g(x)=(t-x)g'(x)+\int\limits_x^t(t-u)g''(u)~du,~~~~g\in C_B^2[0,\infty). 
\end{eqnarray*} 
Now applying the auxiliary operators $\{\O_{n}^{[\alpha]}\}$ over the above relation, we get
\begin{eqnarray*}
\O_{n}^{[\alpha]}(g;x)-g(x)&=& g'(x)\O_{n}^{[\alpha]}(t-x;x)+\O_{n}^{[\alpha]}\left( \int\limits_x^t(t-u)g''(u)~du\right)\\
&=& \mathcal{RL}_{n}^{[\alpha]}\left( \int\limits_x^t(t-u)g''(u)~du\right)- \int\limits_x^{\left(\frac{1+2\phi+2x\beta_n}{2(\psi+\beta_n)} \right)}\left(\frac{1+2\phi+2x\beta_n}{2(\psi+\beta_n)}-u \right)g''(u)~du.
\end{eqnarray*} 
Since, 
\begin{eqnarray*}
\left|\int\limits_x^t(t-u)g''(u)~du\right|\leq\|g''\|\int\limits_x^t|t-u|~du \leq (t-x)^2\|g''\|,
\end{eqnarray*}
and 
\begin{eqnarray*}
\left|\int\limits_x^{\left(\frac{1+2\phi+2x\beta_n}{2(\psi+\beta_n)} \right)}\left(\frac{1+2\phi+2x\beta_n}{2(\psi+\beta_n)}-u \right)g''(u)~du\right|\leq \left(\frac{1+2\phi+2x\beta_n}{2(\psi+\beta_n)}-x \right)^2\|g''\|.
\end{eqnarray*}
Consequently, 
\begin{eqnarray*}
|\O_{n}^{[\alpha]}(g;x)-g(x)|&\leq & \| g''\| \Phi_{n,2}^{\alpha}(x)+\left(\frac{1+2\phi+2x\beta_n}{2(\psi+\beta_n)}-x \right)^2\|g''\|\\
&=& \gamma_n(x)\|g''\|.
\end{eqnarray*}
Also, one can obtain 
\begin{eqnarray*}
|\O_{n}^{[\alpha]}(g;x)|\leq |\mathcal{RL}_{n}^{[\alpha]}(g;x)|+2\|g\|\leq 3\|g\|.
\end{eqnarray*}
Let an auxiliary function $f\in C_B^2[0,\infty)$, and combining the above estimates, we obtain
\begin{eqnarray*}
|\mathcal{RL}_{n}^{[\alpha]}(g;x)-g(x)|&\leq & |\mathcal{RL}_{n}^{[\alpha]}((g-f);x)-(g-f)(x)|+|\mathcal{RL}_{n}^{[\alpha]}(f;x)-f(x)|\\
&&+\left|g(x)- g\left(\frac{1+2\phi+2x\beta_n}{2(\psi+\beta_n)} \right)\right| \\
&\leq & 4\|g-f\|+\gamma_n(x)\|f''\|+\omega\left(g;\left|\frac{1+2\phi-2x\psi}{2(\psi+\beta_n)} \right| \right).
\end{eqnarray*}
Now, taking the infimum over all $f\in C_B^2[0,\infty)$ on the right-hand side of the above inequality and using (\ref{pe1}), we get
\begin{eqnarray*}
|\mathcal{RL}_{n}^{[\alpha]}(g;x)-g(x)|&\leq & M_1 \omega_2(g;\sqrt{\gamma_n(x)})+\omega(g,\eta_n(x)).
\end{eqnarray*}
This completes the proof.
\end{proof}


We now prove the convergence of the proposed operators \eqref{o3} by applying the Bohman-Korovkin theorem \cite{q3}.


\begin{theorem}\label{th1}
Let $g\in C[0,\infty)$, where $C[0,\infty)$ denotes the space of all continuous function defined on $[0,\infty)$,  and $\beta_n\to\infty$, $\alpha\to 0$ as $n\to\infty$. 
Then, the following convergence holds:
\begin{eqnarray}
\underset{\beta_n\to\infty}\lim \mathcal{RL}_{n}^{[\alpha]}(g;x)=g(x),
\end{eqnarray}
uniformly on each $[0,d]\subset[0,\infty)$. 
\end{theorem} 
\begin{proof}
By Lemma \ref{l2}, under the assumptions $\alpha\to 0$, $\beta_n\to\infty$ as $n\to\infty$, we obtain
$$\underset{n\to\infty}\lim \mathcal{RL}_{n}^{[\alpha]}(t^r;x)=x^r,~~~r=0,1,2,$$
uniformly uniformly on every compact interval $[0,d]$ of $[0,\infty)$. Since the operators $\mathcal{RL}_{n}^{[\alpha]}$ are positive and linear, the desired convergence follows immediately from the Bohman-Korovkin theorem (see also \cite{q4}).
\end{proof}
\begin{remark}\label{remark3}
For $n\in\mathbb{N}$, $\max\{\alpha\}=\frac{1}{\beta_n}$ and $0\leq\phi\leq\psi$, there exists a positive constant $M$, such that
\begin{eqnarray}
\mathcal{RL}_{n}^{[\alpha]}((t-x)^2;x)\leq \frac{M}{(\beta_n+\psi)}(1+x)^2.
\end{eqnarray}
\end{remark}
\section{Asymptotic behavior of the operators}\label{sec4}
The following theorem establishes a \emph{Voronovskaya-type asymptotic formula} for the proposed operators, describing the leading term in the approximation error; more precisely, it provides the first-order asymptotic behavior of the approximation error.
\begin{theorem}
Let $g\in C[0,\infty)$ with $g',g''\in C[0,\infty)$ and  $\alpha\to 0$ as $\beta_n\to\infty$. 
Then the following asymptotic formula holds
\begin{eqnarray*}
\underset{\beta_n\to\infty}\lim(\beta_n[\mathcal{RL}_{n}^{[\alpha]}(g;x)-g(x)])=\left(\frac{1}{2}-x \psi +\phi\right)g'(x)+xg''(x).
\end{eqnarray*}
\end{theorem}
\begin{proof}
By Taylor's expansion
\begin{eqnarray}
g(t)=g(x)+g'(x)(t-x)+g''(x)\frac{(t-x)^2}{2!}+\xi(t,x)(t-x)^2,
\end{eqnarray}
where  $\xi(t,x)$ denotes the remainder term satisfying $\underset{t\to x}\lim\xi(t,x)=0$. Applying the operators (\ref{o3}) to both sides of the above expression, we get
\begin{eqnarray*}
\mathcal{RL}_{n}^{[\alpha]}(g(t)-g(x);x)=g'(x)\mathcal{RL}_{n}^{[\alpha]}(t-x;x)+\frac{g''(x)}{2}\mathcal{RL}_{n}^{[\alpha]}((t-x)^2;x)+\mathcal{RL}_{n}^{[\alpha]}(\xi(t,x)(t-x)^2;x).
\end{eqnarray*}
So, with the help of the Cauchy-Schwarz Inequality, we have
\begin{eqnarray}
\mathcal{RL}_{n}^{[\alpha]}(\xi(t,x)(t-x)^2;x)\leq \sqrt{\mathcal{RL}_{n}^{[\alpha]}(\xi^2(t,x);x)}\sqrt{\mathcal{RL}_{n}^{[\alpha]}((t-x)^4;x)}.
\end{eqnarray} 
By Theorem \ref{th1}, we obtain 
\begin{eqnarray}\label{e1}
\underset{\beta_n\to\infty}\lim\mathcal{RL}_{n}^{[\alpha]}(\xi^2(t,x);x)=\xi^2(x,x)=0.
\end{eqnarray}
Combining (\ref{e1}) with Lemma \ref{l1}, we get
\begin{eqnarray}\label{e2}
\underset{\beta_n\to\infty}\lim\{\beta_n^2\mathcal{RL}_{n}^{[\alpha]}((t-x)^4;x)&=& 12x^2,
\end{eqnarray}
and
\begin{eqnarray*}
\underset{\beta_n\to\infty}\lim\{\beta_n\mathcal{RL}_{n}^{[\alpha]}((t-x)^2;x)&=& 2x.
\end{eqnarray*}
Therefore, by Lemma \ref{l1}, we have
\begin{eqnarray*}
\underset{\beta_n\to\infty}\lim(\beta_n[\mathcal{RL}_{n}^{[\alpha]}(g;x)-g(x)])=\left(\frac{1}{2}-x \psi +\phi\right)g'(x)+xg''(x).
\end{eqnarray*}
Hence proved.
\end{proof}

\section{Convergence result in weighted spaces}\label{sec5}
This section is devoted to the convergence results of the operators (\ref{o3}) in the \emph{weight space}. Further details on {weighted spaces} can be found in \cite{q5,q6}. Additionally, properties related to weighted approximation are explored in \cite{ZDX1}.\\
 
Let $B_w[0,\infty)=\{g:[0,\infty)\to\mathbb{R} |~~  |g(x)|\leq Mw(x)\}$,  where $M>0$ is a constant. Along with this, the supremum norm is defined by $$\|g\|_w=\underset{x\in [0,\infty)}\sup\frac{g(x)}{w(x)}<+\infty.$$ 
Also, define the spaces
$$C_w[0,\infty)=\{g\in B_w[0,\infty), ~g~\text{is~continuous} \},$$
$$C_w^k[0,\infty)=\{g\in C_w[0,\infty),\underset{x\to\infty}\lim\frac{|g(x)|}{w(x)}=k_g<+\infty\},$$      
where $w(x)=1+x^2$ is a weight function. 
The following lemma will be used to prove the main theorem.
\begin{lemma}\cite{G2}
Let $L_n:C_w[0,\infty)\to B_w[0,\infty)$ such that $\underset{n\to\infty}\lim\|L_n(t^r;x)-x^r\|_w=0,~~r=0,1,2$. Then, for every $g\in C_w^k[0,\infty)$, we have 
$$\underset{n\to\infty}\lim\|L_ng-g\|_w=0.$$
\end{lemma}
\begin{theorem}
Let $\{\mathcal{RL}_{n}^{[\alpha]}g\}$ be the sequence of positive linear operators defined by (\ref{o3}). Then the following convergence holds:
\begin{eqnarray*}
\underset{\beta_n\to\infty}\lim \|\mathcal{RL}_{n}^{[\alpha]}g-g\|_w=0,~~\text{for}~g\in C_w^k[0,\infty).
\end{eqnarray*}
\end{theorem}

\begin{proof}
    It is sufficient to show that $\underset{\beta_n\to\infty}\lim \|\mathcal{RL}_{n}^{[\alpha]}(t^r;x)-x^r\|_w=0$ holds for $r=0,1,2$ to prove the above theorem. Clearly, 
\begin{eqnarray}
\underset{\beta_n\to\infty}\lim \|\mathcal{RL}_{n}^{[\alpha]}(1;x)-1\|_w=0.
\end{eqnarray}
 By Lemma \ref{l2}, we have
 \begin{eqnarray*}
\|\mathcal{RL}_{n}^{[\alpha]}(t;x)-x\|_w &=& \underset{x\geq 0}\sup  \frac{\left|\frac{1+2\phi+2x\beta_n}{2(\psi+\beta_n)}-x\right|}{1+x^2}\leq\left(\frac{1+2\phi}{2(\psi+\beta_n)}\right)\underset{x\geq 0}\sup\frac{1}{1+x^2}+\left(\frac{\beta_n}{(\psi+\beta_n)}-1\right)\underset{x\geq 0}\sup\frac{x}{1+x^2},
 \end{eqnarray*}
 implies that $\|\mathcal{RL}_{n}^{[\alpha]}(t;x)-x\|_w \to 0$ as $\beta_n\to\infty$. Similarly, 
 \begin{eqnarray*}
 \|\mathcal{RL}_{n}^{[\alpha]}(t^2;x)-x^2\|_w &=& \underset{x\geq 0}\sup \frac{\left|\frac{ 1+3 \phi +3 \phi ^2+6 x \beta _n+6 x \phi  \beta _n+3 \alpha x \beta _n^2+3 x^2 \beta _n^2}{3\left(\psi +\beta _n\right)^2}-x^2\right|}{1+x^2}\\
 &\leq & \frac{1+3 \phi +3 \phi ^2}{3\left(\psi +\beta _n\right)^2}\underset{x\geq 0}\sup\frac{1}{1+x^2}+\frac{2\beta_n+2\phi\beta_n+3\alpha\beta_n^2}{\left(\psi +\beta _n\right)^2}\underset{x\geq 0}\sup\frac{1}{1+x^2}+\left(\frac{\beta_n^2}{\left(\psi +\beta _n\right)^2}-1\right)\underset{x\geq 0}\sup\frac{x^2}{1+x^2},
 \end{eqnarray*}
shows that $\|\mathcal{RL}_{n}^{[\alpha]}(t^2;x)-x^2\|_w\to 0$ as $\beta_n\to\infty$. 

This completes the proof.
\end{proof}

The following theorem establishes the weighted convergence of the proposed operators in the weighted space $C_w^k[0,\infty)$.
\begin{theorem}
Let $g\in C_w^k[0,\infty)$, $l>0$, and suppose $\alpha=\alpha(n)\to 0$ as $\beta_n\to\infty$. Then  the following relation holds:
\begin{eqnarray}
\underset{\beta_n\to\infty}\lim \underset{x\geq0}\sup \frac{\left|\mathcal{RL}_{n}^{[\alpha]}(g;x)-g(x)\right|}{(1+x^2)^{1+l}}=0,~~~x\geq 0.
\end{eqnarray}
\end{theorem}
\begin{proof}
Let $x_0$ be a fixed point. Then
\begin{eqnarray}\label{eq1}
\nonumber\underset{x\geq0}\sup \frac{\left|\mathcal{RL}_{n}^{[\alpha]}(g;x)-g(x)\right|}{(1+x^2)^{1+l}}&\leq & \underset{x\leq x_0}\sup \frac{\left|\mathcal{RL}_{n}^{[\alpha]}(g;x)-g(x)\right|}{(1+x^2)^{1+l}}+\underset{x>x_0}\sup \frac{\left|\mathcal{RL}_{n}^{[\alpha]}(g;x)-g(x)\right|}{(1+x^2)^{1+l}}\\
\nonumber&\leq & \|\mathcal{RL}_{n}^{[\alpha]}(g;x)-g(x)\|_{C[0,x_0]}+\|g\|_{w} \underset{x>x_0}\sup \frac{\left|\mathcal{RL}_{n}^{[\alpha]}((1+t^2);x)\right|}{(1+x^2)^{1+l}} +\underset{x>x_0}\sup\frac{|g|}{(1+x^2)^{1+l}}\\
&=& B_1+B_2+B_3~(say). 
\end{eqnarray}
Here,
\begin{eqnarray*}
B_3=\underset{x>x_0}\sup\frac{|g|}{(1+x^2)^{1+l}}\leq \frac{\|g\|_{w}}{(1+x_0^2)^l},~~\text{(as~$|g(x)|\leq M(1+x^2))$}.
\end{eqnarray*}
Therefore, choosing $x_0$ sufficiently large, and for an arbitrary $\epsilon>0$, we have
\begin{eqnarray}\label{eq2}
B_3=\frac{\|g\|_{w}}{(1+x_0^2)^l}\leq \frac{\epsilon}{3}.
\end{eqnarray}
By the convergence established in the previous theorem, we have $$\underset{\beta_n\to\infty}\lim \underset{x>x_0}\sup \frac{\left|\mathcal{RL}_{n}^{[\alpha]}((1+t^2);x)\right|}{(1+x^2)}=1.$$ 
There exists $n_1\in\mathbb{N}$, and for any $\epsilon>0$, we have
\begin{eqnarray}\label{eq3}
B_2=\| g\|_w\underset{x>x_0}\sup \frac{\left|\mathcal{RL}_{n}^{[\alpha]}((1+t^2);x)\right|}{(1+x^2)^{1+l}}\leq \frac{\| g\|_w}{(1+x^2)^l}\leq \frac{\| g\|_w}{(1+x_0^2)^l}<\frac{\epsilon}{3}.
\end{eqnarray}
By Theorem \ref{th1}
\begin{eqnarray}\label{eq4}
B_1=\|\mathcal{RL}_{n}^{[\alpha]}(g;x)-g(x)\|_{C[0,x_0]}\|\leq \frac{\epsilon}{3}.
\end{eqnarray}
Combining (\ref{eq2}-\ref{eq4}) with (\ref{eq1}), we obtain our required result.
\end{proof}


\section{Quantitative Approximation}\label{sec6}
To investigate the rate of convergence of the sequence of positive linear operators in the weighted space $C_w^k[0,\infty)$, Ispir \cite{IN1} introduced the weighted modulus of continuity, which is defined as: 
\begin{eqnarray}
\Delta(g;\delta)=\underset{0\leq h\leq\delta,~0\leq x\leq\infty}\sup \frac{|g(x+h)-g(x)|}{(1+h^2)(1+x^2)},~~~~~~~,\delta>0,~~   g\in C_w^k[0,\infty). 
\end{eqnarray} 
There exists a fixed number $x_0$ such that
\begin{eqnarray*}
\Delta(g;\delta)&\leq & \underset{0\leq x\leq x_0, 0\leq h\leq \delta}\sup|g(x+h)-g(x)|+\underset{x_0\leq x\leq \infty}\sup\left|\frac{g(x+h)}{1+(x+h)^2}-k \right|+\delta k\underset{x_0\leq x\leq \infty}\sup\left(\frac{2x+\delta}{1+x^2} \right)+\underset{x_0\leq \infty}\sup\left| \frac{g(x)}{1+x^2}-k\right|\\
&<& \omega(g;\delta)+2\delta k+\epsilon,
\end{eqnarray*}
where $\omega(g;\delta)$ denotes the usual modulus of continuity on  $[0,x_0]$.
Consequently 
$$\underset{\delta\to 0}\lim\Delta(g;\delta)=0.$$
It follows immediately that, $\Delta(g;\lambda\delta)\leq2(1+\delta^2)(1+\lambda)\Delta(g;\delta),~~\lambda>0$. 
Using the definition of the weighted modulus of continuity together with the above inequality, the following estimate holds:
 \begin{eqnarray}
 \nonumber|g(t)-g(x)|&\leq &(1+x^2)(1+(t-x)^2)\Delta(g;|t-x|)\\
 &\leq & 2\left(1+\frac{|t-x|}{\delta}\right)(1+\delta^2)(1+(t-x)^2)(1+x^2)\Delta(g;|t-x|).
 \end{eqnarray}
The weighted modulus of continuity enables us to establish the following quantitative Voronovskaya-type theorem. 
\subsection{Quantitative Voronovskaya-type theorem}
\begin{theorem}\label{th2}
Let $g\in C_w^k[0,\infty)$ be twice continuously differentiable with $g', g''\in C_w^k[0,\infty)$. For all sufficiently large $\beta_n$, we obtain the following result:
\begin{eqnarray*}
\beta_n\left|\mathcal{RL}_{n}^{[\alpha]}(g;x)-g(x)-g'(x)\mathcal{RL}_{n}^{[\alpha]}((t-x);x)-\frac{g''(x)}{2!}\mathcal{RL}_{n}^{[\alpha]}((t-x)^2;x) \right|=O(1)\Delta\left(g'';\sqrt{\frac{1}{\beta_n}}\right).
\end{eqnarray*} 
\end{theorem}
\begin{proof}
By Taylor's expansion
\begin{eqnarray}
g(t)-g(x)=g'(x)(t-x)+\frac{g''(x)}{2}(t-x)^2+\eta(t,x),
\end{eqnarray}
where $\eta(t,x)=\frac{g''(\zeta)-g''(x)}{2!}(\zeta-x)^2$ and $\zeta\in (t,x)$.
Applying the operator (\ref{o3}) to both sides
\begin{eqnarray}\label{n1}
\beta_n\left|\mathcal{RL}_{n}^{[\alpha]}(g;x)-g(x)-g'(x)\mathcal{RL}_{n}^{[\alpha]}((t-x);x)-\frac{g''(x)}{2}\mathcal{RL}_{n}^{[\alpha]}((t-x)^2;x)\right|\leq \beta_n\mathcal{RL}_{n}^{[\alpha]}(|\eta(t,x)|;x).
\end{eqnarray}
Now using the property of weighted modulus of continuity, we obtain
\begin{eqnarray*}
\frac{g''(\zeta)-g''(x)}{2} &\leq &\frac{1}{2}(1+(t-x)^2)(1+x^2)\Delta(g'',|t-x|)\\
&\leq& \left(1+\frac{|t-x|}{\delta} \right)(1+\delta^2)(1+(t-x)^2)(1+x^2)\Delta(g'',\delta).
\end{eqnarray*}
Consequently,
\begin{eqnarray}
\left|\frac{g''(\zeta)-g''(x)}{2}\right| &\leq &  
\begin{cases}
    2(1+\delta^2)^2(1+x^2)\Delta(g'',\delta),& |t-x|<\delta,\\
    2(1+\delta^2)^2(1+x^2)\frac{(t-x)^4}{\delta^4}\Delta(g'',\delta),& |t-x|\geq\delta.
\end{cases} 
\end{eqnarray}
Since $0<\delta<1$
\begin{eqnarray}
\left|\frac{g''(\zeta)-g''(x)}{2}\right| &\leq & 8(1+x^2)\left(1+\frac{(t-x)^4}{\delta^4}\right)\Delta(g'',\delta). 
\end{eqnarray}
Hence, $$|\eta(t,x)|\leq 8(1+x^2)\left((t-x)^2+\frac{(t-x)^6}{\delta^4}\right)\Delta(g'',\delta).$$
Applying the Lemma \ref{l1}, we have
\begin{eqnarray*}
\mathcal{RL}_{n}^{[\alpha]}(| \zeta(t,x)|;x)&\leq & 8(1+x^2)\Delta(g'',\delta)\left(\mathcal{RL}_{n}^{[\alpha]}((t-x)^2;x)+\frac{\mathcal{RL}_{n}^{[\alpha]}((t-x)^6;x)}{\delta^4}\right) \\
&\leq & 8(1+x^2)\Delta(g'',\delta) \left(O\left(\frac{1}{\beta_n} \right)+\frac{1}{\delta^4} O\left(\frac{1}{\beta_n^3} \right) \right),~~\text{as}~\beta_n\to\infty.
\end{eqnarray*}
Choosing, $\delta=\sqrt{\frac{1}{\beta_n}}$ and then
\begin{eqnarray}
\mathcal{RL}_{n}^{[\alpha]}(| \zeta(t,x)|;x)\leq 8(1+x^2)\left( O\left(\frac{1}{\beta_n} \right)\right)\Delta\left(g'',\sqrt{\frac{1}{\beta_n}}\right).
\end{eqnarray}
Thus, we obtain
\begin{eqnarray}\label{n2}
\beta_n\mathcal{RL}_{n}^{[\alpha]}(|\eta(t,x)|;x)=O(1)\Delta\left(g'', \sqrt{\frac{1}{\beta_n}}\right).
\end{eqnarray}
By (\ref{n1}) and (\ref{n2}), the desired estimate is established.
\end{proof}
\subsection{Gr$\ddot{\text{u}}$ss Voronovskaya-type Theorem}
In 1935, Gr$\ddot{\text{u}}$ss \cite{GG}, defined an inequality, known as \emph{Gr$\ddot{\text{u}}$ss type inequality} after his name. This inequality 
plays an important role in approximation theory and has numerous applications in the study of positive linear operators. The application of  Gr$\ddot{\text{u}}$ss inequality can be found in \cite{AGR}, 
and some results are discussed in \cite{GT}.
Furthermore, using the Gr$\ddot{\text{u}}$ss inequality, Gal and Gonska \cite{GG1} estimated a Voronovskaya-type theorem for the Bernstein operators, known as \emph{Gr$\ddot{\text{u}}$ss-Voronovskaya type} theorem. Motivated by these results, we establish the following Gr$\ddot{\text{u}}$ss-Voronovskaya-type theorem for the proposed operators.


\begin{theorem}
Let $f, g\in C_w^k[0,\infty)$ be twice continuously differentiable together with $f', f'', g', g''\in C_w^k[0,\infty)$. Then the following limit holds:
\begin{eqnarray*}
\underset{\beta_n\to\infty}\lim\beta_n\left(\mathcal{RL}_{n}^{[\alpha]}(fg;x)-\mathcal{RL}_{n}^{[\alpha]}(f;x)\mathcal{RL}_{n}^{[\alpha]}(g;x)\right)&=& 2xf'(x)g'(x).
\end{eqnarray*}
\end{theorem}
\begin{proof}
Using the identity
\begin{eqnarray*}
    (fg)'=f'g+fg' \quad (fg)''=f''g+2f'g'+fg''
\end{eqnarray*}
together with a simple algebraic decomposition, we obtain
\begin{eqnarray}\label{en4}
\nonumber\beta_n\left(\mathcal{RL}_{n}^{[\alpha]}(fg;x)-\mathcal{RL}_{n}^{[\alpha]}(f;x)\mathcal{RL}_{n}^{[\alpha]}(g;x)\right)&=&\beta_n\Bigg(\Bigg(\mathcal{RL}_{n}^{[\alpha]}(fg;x)-f(x)g(x)-(fg)'(x)\mathcal{RL}_{n}^{[\alpha]}((t-x);x)\\\nonumber
&&-\frac{(fg)''(x)}{2!}\mathcal{RL}_{n}^{[\alpha]}((t-x)^2;x)\Bigg)-g(x)\Bigg(\mathcal{RL}_{n}^{[\alpha]}(f;x)-f(x)\\\nonumber
&&-f'(x)\mathcal{RL}_{n}^{[\alpha]}((t-x);x)-\frac{f''(x)}{2!}\mathcal{RL}_{n}^{[\alpha]}((t-x)^2;x)\Bigg)\\\nonumber
&&-\mathcal{RL}_{n}^{[\alpha]}(f;x)\Bigg(\mathcal{RL}_{n}^{[\alpha]}(g;x)-g(x)-g'(x)\mathcal{RL}_{n}^{[\alpha]}((t-x);x)\\\nonumber
&&-\frac{g''(x)}{2!}\mathcal{RL}_{n}^{[\alpha]}((t-x)^2;x)\Bigg)+\frac{g''(x)}{2!}\mathcal{RL}_{n}^{[\alpha]}((t-x)^2;x)\\\nonumber
&&\times\Bigg(f(x)-\mathcal{RL}_{n}^{[\alpha]}(f;x)  \Bigg)+f'(x)g'(x)\mathcal{RL}_{n}^{[\alpha]}((t-x)^2;x)\\
&&+g'(x)\mathcal{RL}_{n}^{[\alpha]}((t-x);x)\left(f(x)-\mathcal{RL}_{n}^{[\alpha]}((f;x) \right) \Bigg)
\end{eqnarray} 
Letting $\beta_n\to\infty$ (and hence, $\alpha\to 0 $), and applying  Theorems \ref{th1} and \ref{th2} together with Lemma \ref{l1}, we obtain
\begin{eqnarray*}
\underset{\beta_n\to\infty}\lim\beta_n\left(\mathcal{RL}_{n}^{[\alpha]}(fg;x)-\mathcal{RL}_{n}^{[\alpha]}(f;x)\mathcal{RL}_{n}^{[\alpha]}(g;x)\right)&=& 2xf'(x)g'(x).
\end{eqnarray*}
This completes the proof.
\end{proof}


\section{Rate of convergence by means of the function of bounded variation}\label{sec7}
This section deals with the upper bound in the approximation in terms of the function with the derivative of \emph{bounded variation}. The approximation of functions with derivatives of bounded variation has been studied extensively. Early contributions include the work of Bojani\'c \cite{BR}, Bojani\'c and Vuilleumier \cite{BRV}, and Cheng \cite{CF1,CF2}. Later, Guo \cite{GS} employed the \emph{Berry–Esseen theorem} to obtain convergence estimates, while Bojani\'c and Khan \cite{BRK}, Shaw et al. \cite{SLL}, and others further developed the rate of convergence for positive linear operators acting on functions with derivatives of bounded variation.

We now establish the rate of convergence of the proposed operators for functions with derivatives of bounded variation.
Let $DBV[0,\infty)$ denote the space of all continuous functions whose derivatives are of bounded variation on every finite subinterval of $[0,\infty)$. 
One can observe that for each $g\in DBV[0,\infty)$, it admits
\begin{eqnarray}
g(x)=\int\limits_0^x h(s)~ds+g(0),
\end{eqnarray}
where $h$ is a function of bounded variation on each finite sub-interval of $[0,\infty)$. 
Now, define the auxiliary function $g_x$ such that
\begin{eqnarray}\label{eq5}
g_x(t) &= &  
\begin{cases}
    g(t)-g(x-),& 0\leq t<x,\\
    0,& t=x,\\
    g(t)-g(x+), &  x<t<\infty.
\end{cases} 
\end{eqnarray}
Let $V_a^b g$ denotes total variation of a real valued function $g$ on interval $[a,b]\subset[0,\infty)$ with 
\begin{eqnarray}
V_a^b g=\underset{\mathcal{P}}\sup\left(\sum\limits_{k=0}^{n_{P}-1}|g(x_{k+1})-g(x_k)| \right),
\end{eqnarray} 
where, $\mathcal{P}$ is the set of all partitions $P=\{a=x_0,\cdots,x_{n_P}=b\}$ of the interval $[a,b]$. 

\begin{lemma}\label{l3}
For every $x\geq 0$ and sufficiently large of $\beta_n$, there exists a positive constant $M>0$, the following estimates hold:
\begin{enumerate}
\item{} $J_n^{[\alpha]}(x;y)= \int\limits_0^y \mathcal{Y}_n^{[\alpha]}(x;t)~dt\leq \frac{M}{(x-y)^2}\frac{(1+x)^2}{(\beta_n+\psi)},~~0\leq y<x,$
\item{} $1-J_n^{[\alpha]}(x;z)= \int\limits_z^\infty \mathcal{Y}_n^{[\alpha]}(x;t)~dt\leq \frac{M}{(z-x)^2}\frac{(1+x)^2}{(\beta_n+\psi)},~~x<z<\infty.$
\end{enumerate}
\end{lemma}
\begin{proof}
For $0\leq t\leq y<x$, we have $x-t\geq x-y$, we have
\begin{eqnarray*}
    1\leq \left(\frac{x-t}{x-y}\right)^2. 
\end{eqnarray*}
Therefore,
\begin{eqnarray*}
\int\limits_0^y \mathcal{Y}_n^{[\alpha]}(x;t)~dt &\leq & \int\limits_0^y \left(\frac{x-t}{x-y} \right)^2\mathcal{Y}_n^{[\alpha]}(x;t)~dt\\
&\leq &\frac{1}{(x-y)^2}\mathcal{RL}_{n}^{[\alpha]}((t-x)^2;x)\\
&\leq & \frac{M}{(x-y)^2}\frac{(1+x)^2}{(\beta_n+\psi)}.
\end{eqnarray*}
Similarly, we can prove the second result.
\end{proof}
\begin{theorem}
Let $g\in DBV[0,\infty)$. Then for every $x\in(0,\infty)$, for sufficiently large values of $\beta_n$, the following estimate holds.
\begin{eqnarray*}
|\mathcal{RL}_{n}^{[\alpha]}(g;x)-g(x)|&\leq & \frac{1}{4}|g'(x+)+g'(x-)| \left|\frac{1+2\phi-2x\psi}{(\psi+\beta_n)}\right|+\frac{1}{2}|(g'(x+)-g'(x-))|\sqrt{\frac{M}{(\beta_n+\psi)}}(1+x)\nonumber\\
&& + M\frac{(1+x)^2}{x(\beta_n+\psi)} \sum\limits_{i=1}^{[\sqrt{\beta_n}]}\left(V_{x-\frac{x}{i}}^x g'_x\right)+ \left(V_{x-\frac{x}{\sqrt{\beta_n}}}^x g'_x \right)\frac{x}{\sqrt{\beta_n}}+(V_x^{x+\frac{x}{\sqrt{\beta_n}}} g'_x) \frac{x}{\sqrt{\beta_n}}\\
&& + \frac{M(1+x)^2}{x(\beta_n+\psi)} \sum\limits_{v=1}^{[\sqrt{\beta_n}]}(V_x^{x+\frac{x}{v}} g'_x).
\end{eqnarray*}
\end{theorem}

\begin{proof}
Using the hypothesis (\ref{eq5}), one can write as:
\begin{eqnarray}\label{bv1}
g'(s)&=& g_x'(s)+\frac{1}{2}(g'(x+)+g'(x-))+\frac{1}{2}(g'(x+)-g'(x-))\operatorname{sgn}(s-x)(s-x)\nonumber\\
&&+\tau_x(s)\{g'(s)-\frac{1}{2}(g'(x+)+g'(x-)) \},
\end{eqnarray}
where $\tau_x(s)$
\begin{eqnarray}\label{bv2}
\tau_x(s)=
\begin{cases}
1, & s=x\\
0, & s\neq x.
\end{cases}
\end{eqnarray}
By equation (\ref{o4}), we can write as
\begin{eqnarray}\label{bv3}
\mathcal{RL}_{n}^{[\alpha]}(g;x)-g(x)&=& \int\limits_0^\infty \mathcal{Y}_n^{[\alpha]}(x;s)(g(s)-g(x))~ds\nonumber\\
&=&  \int\limits_0^\infty \mathcal{Y}_n^{[\alpha]}(x;s)\left(\int\limits_x^s g'(t)~dt\right)ds.
\end{eqnarray}
Here, it is clear that
\begin{eqnarray*}
\int\limits_x^s \tau_x(t)~dt=0.
\end{eqnarray*}
Therefore
\begin{eqnarray}\label{bv4}
\int\limits_0^\infty \mathcal{Y}_n^{[\alpha]}(x;s)\int\limits_x^s\left(\tau_x(t)\{g'(t)-\frac{1}{2}(g'(x+)+g'(x-))\}dt\right)ds=0.
\end{eqnarray}
Also by (\ref{o4}), one can obtains
\begin{eqnarray}\label{bv5}
\int\limits_0^\infty \mathcal{Y}_n^{[\alpha]}(x;s)\left(\int\limits_x^s\frac{1}{2}(g'(x+)+g'(x-))~dt\right)ds &=&\frac{1}{2}(g'(x+)+g'(x-)) \int\limits_0^\infty \mathcal{Y}_n^{[\alpha]}(x;s)(s-x)~ds\nonumber\\&=& \frac{1}{2}(g'(x+)+g'(x-)) \mathcal{RL}_{n}^{[\alpha]}((s-x);x).
\end{eqnarray}
Moreover, 
\begin{eqnarray}\label{bv6}
\left|\int\limits_0^\infty \mathcal{Y}_n^{[\alpha]}(x;s)\left(\frac{1}{2}\int\limits_x^s(g'(x+)-g'(x-))\text{sgn}(t-x)~dt\right)ds\right| &\leq & \frac{1}{2}|(g'(x+)-g'(x-))|\int\limits_0^\infty \mathcal{Y}_n^{[\alpha]}(x;s)|s-x|~ds\nonumber\\
&\leq & \frac{1}{2}|(g'(x+)-g'(x-))| \mathcal{RL}_{n}^{[\alpha]}(|s-x|;x)\nonumber\\
&\leq & \frac{1}{2}|(g'(x+)-g'(x-))|\left(\Phi_{n,2}^{\alpha}(x) \right)^{\frac{1}{2}}.
\end{eqnarray}
By Lemma \ref{l3} and Remark \ref{remark3}. Then by (\ref{bv1}), we obtain
\begin{eqnarray}\label{bv7}
|\mathcal{RL}_{n}^{[\alpha]}(g;x)-g(x)|&\leq & \frac{1}{2}|g'(x+)+g'(x-)| |\Phi_{n,1}^{\alpha}(x)|+\frac{1}{2}|(g'(x+)-g'(x-))|\left(\Phi_{n,2}^{\alpha}(x) \right)^{\frac{1}{2}}\nonumber\\
&& + \left| \int\limits_0^\infty \mathcal{Y}_n^{[\alpha]}(x;s)\left( \int\limits_x^s g_x'(t)~dt\right)ds\right|.
\end{eqnarray} 
\begin{eqnarray}\label{bv8}
\int\limits_0^\infty \mathcal{Y}_n^{[\alpha]}(x;s)\left( \int\limits_x^s g_x'(t)~dt\right)ds &=& \int\limits_0^x \mathcal{Y}_n^{[\alpha]}(x;s)\left( \int\limits_x^s g_x'(t)~dt\right)ds+\int\limits_x^\infty \mathcal{Y}_n^{[\alpha]}(x;s)\left( \int\limits_x^s g_x'(t)~dt\right)ds\nonumber\\
&=& I_1+I_2~(\text{say}),
\end{eqnarray}
where $I_1=\int\limits_0^x \mathcal{Y}_n^{[\alpha]}(x;s)\left( \int\limits_x^s g_x'(t)~dt\right)ds$, $I_2=\int\limits_x^\infty \mathcal{Y}_n^{[\alpha]}(x;s)\left( \int\limits_x^s g_x'(t)~dt\right)ds$.
So, 
\begin{eqnarray*}
I_1&=&\int\limits_0^x\left( \int\limits_x^s g_x'(t)~dt\right)\mathcal{Y}_n^{[\alpha]}(x;s)~ds\\
&=& \int\limits_0^x\left( \int\limits_x^s g_x'(t)~dt\right)\frac{\partial}{\partial s} J_n^{[\alpha]}(x;s)ds\\
&=& \int\limits_0^x g_x'(s)J_n^{[\alpha]}(x;s)ds\\
&=& \left(\int\limits_0^y+\int\limits_y^x\right)g_x'(s)J_n^{[\alpha]}(x,s)ds.
\end{eqnarray*}
For sufficiently large $\beta_n$, with $\beta_n\geq 1$, we set $y=x-\frac{x}{\sqrt{\beta_n}}$,  we have 
\begin{eqnarray*}
\left|\int\limits_{x-\frac{x}{\sqrt{\beta_n}}}^x g_x'(s)J_n^{[\alpha]}(x,s)ds\right|&\leq &\int\limits_{x-\frac{x}{\sqrt{\beta_n}}}^x |g_x'(s)|J_n^{[\alpha]}(x,s)ds\\
&\leq & \int\limits_{x-\frac{x}{\sqrt{\beta_n}}}^x |g_x'(s)-g_x'(x)|ds,~~\text{as}~g_x'(x)=0,~\text{and}~|J_n^{[\alpha]}(x,s)|\leq 1\\
&\leq & \int\limits_{x-\frac{x}{\sqrt{\beta_n}}}^x V_s^x g'_x~ ds \leq   V_{x-\frac{x}{\sqrt{\beta_n}}}^x g'_x\int\limits_{x-\frac{x}{\sqrt{\beta_n}}}^x ds=\left(V_{x-\frac{x}{\sqrt{\beta_n}}}^x g'_x \right)\frac{x}{\sqrt{\beta_n}}.
\end{eqnarray*}
By Lemma \ref{l3} as well as taking $l=\frac{x}{x-s}$, it hold
\begin{eqnarray*}
\int\limits_0^{x-\frac{x}{\sqrt{\beta_n}}}|g_x'(s)|J_n^{[\alpha]}(x,s)ds &\leq & M\frac{(1+x)^2}{(\beta_n+\psi)}\int\limits_0^{x-\frac{x}{\sqrt{\beta_n}}}\frac{|g'_x(s)|}{(x-s)^2}~ds\\
&\leq & M\frac{(1+x)^2}{(\beta_n+\psi)}\int\limits_0^{x-\frac{x}{\sqrt{\beta_n}}} \left(V_s^x g'_x\right)\frac{1}{(x-s)^2}~ds\\
&=& M\frac{(1+x)^2}{x(\beta_n+\psi)}\int\limits_1^{\sqrt{\beta_n}}\left(V_{x-\frac{x}{l}}^x g'_x\right)~dl\leq M\frac{(1+x)^2}{x(\beta_n+\psi)} \sum\limits_{i=1}^{[\sqrt{\beta_n}]}\left(V_{x-\frac{x}{i}}^x g'_x\right).
\end{eqnarray*}
Thus, 
\begin{eqnarray}
|I_1|\leq M\frac{(1+x)^2}{x(\beta_n+\psi)} \sum\limits_{i=1}^{[\sqrt{\beta_n}]}\left(V_{x-\frac{x}{i}}^x g'_x\right)+ \left(V_{x-\frac{x}{\sqrt{\beta_n}}}^x g'_x \right)\frac{x}{\sqrt{\beta_n}}
\end{eqnarray}
Now, $I_2$ can be written as given below, and integrating by parts
\begin{eqnarray*}
|I_2|&=& \left|\int\limits_x^{z} \left(\int\limits_x^s g'_x(t)~dt \right)\frac{\partial}{\partial s}(1-J_n^{[\alpha]}(x,s))ds+\int\limits_z^{\infty} \left(\int\limits_x^s g'_x(t)~dt \right)\frac{\partial}{\partial s}(1-J_n^{[\alpha]}(x,s))ds\right|\\
&=& \Bigg|\int\limits_x^{z}g'_x(t)\cdot (1-J_n^{[\alpha]}(x,z))~dt-\int\limits_x^z g'_x(s)(1-J_n^{[\alpha]}(x,s))~ds +\Bigg[\int\limits_x^s g'_x(t)~dt~(1-J_n^{[\alpha]}(x,s))\Bigg]_z^{\infty}\\
&&-\int\limits_z^{\infty}g'_x(s)(1-J_n^{[\alpha]}(x,s))~ds\Bigg|\\
&\leq & \Bigg|\int\limits_x^z g'_x(s)(1-J_n^{[\alpha]}(x,s))~ds \Bigg|+\Bigg| \int\limits_z^{\infty}g'_x(s)(1-J_n^{[\alpha]}(x,s))~ds\Bigg|.
\end{eqnarray*}
Applying Lemma \ref{l3}, and taking $z=x+\frac{x}{\sqrt{\beta_n}}$, we get
\begin{eqnarray*}
|I_2|&\leq & \frac{M(1+x)^2}{(\beta_n+\psi)}\int\limits_z^{\infty} (V_x^s g'_x)~\frac{ds}{(s-x)^2}+\int\limits_x^z(V_x^s g'_x)~ds\\
&\leq & \frac{M(1+x)^2}{(\beta_n+\psi)}\int\limits_{x+\frac{x}{\sqrt{\beta_n}}}^{\infty} (V_x^s g'_x)~\frac{ds}{(s-x)^2}+\int\limits_x^{x+\frac{x}{\sqrt{\beta_n}}}(V_x^{x+\frac{x}{\sqrt{\beta_n}}} g'_x)~ds\\
&=& \frac{M(1+x)^2}{(\beta_n+\psi)}\int\limits_{x+\frac{x}{\sqrt{\beta_n}}}^{\infty} (V_x^s g'_x)~\frac{ds}{(s-x)^2}+ (V_x^{x+\frac{x}{\sqrt{\beta_n}}} g'_x) \frac{x}{\sqrt{\beta_n}}.
\end{eqnarray*}
If we consider, $s=x+\frac{x}{r}$, then we get
\begin{eqnarray*}
\int\limits_{x+\frac{x}{\sqrt{\beta_n}}}^{\infty} (V_x^s g'_x)~\frac{ds}{(s-x)^2}&=& \int\limits_0^{\sqrt{\beta_n}}(V_x^{x+\frac{x}{r}} g'_x)\frac{dr}{x}\\
&\leq & \frac{1}{x}\sum\limits_{v=1}^{[\sqrt{\beta_n}]}\int\limits_v^{v+1}(V_x^{x+\frac{x}{r}} g'_x)~dr\\
&\leq & \frac{1}{x}\sum\limits_{v=1}^{[\sqrt{\beta_n}]}(V_x^{x+\frac{x}{v}} g'_x).
\end{eqnarray*}
So,
\begin{eqnarray*}
|I_2| \leq  (V_x^s g'_x) \frac{x}{\sqrt{\beta_n}} + \frac{M(1+x)^2}{x(\beta_n+\psi)} \sum\limits_{v=1}^{[\sqrt{\beta_n}]}(V_x^{x+\frac{x}{v}} g'_x).
\end{eqnarray*}
Injecting the value of $I_1$ and $I_2$ in equation (\ref{bv8}), and using (\ref{bv8}), we get our required result.
\begin{eqnarray*}
|\mathcal{RL}_{n}^{[\alpha]}(g;x)-f(x)|&\leq & \frac{1}{4}|g'(x+)+g'(x-)| \left|\frac{1+2\phi-2x\psi}{(\psi+\beta_n)}\right|+\frac{1}{2}|(g'(x+)-g'(x-))|\sqrt{\frac{M}{(\beta_n+\psi)}}(1+x)\nonumber\\
&& + M\frac{(1+x)^2}{x(\beta_n+\psi)} \sum\limits_{i=1}^{[\sqrt{\beta_n}]}\left(V_{x-\frac{x}{i}}^x g'_x\right)+ \left(V_{x-\frac{x}{\sqrt{\beta_n}}}^x g'_x \right)\frac{x}{\sqrt{\beta_n}}+(V_x^{x+\frac{x}{\sqrt{\beta_n}}} g'_x) \frac{x}{\sqrt{\beta_n}}\\
&& + \frac{M(1+x)^2}{x(\beta_n+\psi)} \sum\limits_{v=1}^{[\sqrt{\beta_n}]}(V_x^{x+\frac{x}{v}} g'_x).
\end{eqnarray*}
Thus, the proof is completed.
\end{proof}

\section{Tabular and graphical representation}\label{sec8}

This section illustrates the convergence behaviour of the proposed operators (\ref{o3}) through graphical and numerical examples. A comparative study with the operators defined in (\ref{o2}) is also presented.
\begin{example}
Consider the function $g(x)=\sin{x}$. Let $\alpha=\frac{1}{50}$, $\phi=0.1$ and $~\psi=0.3$. Figure \ref{F1} shows the approximations obtained by the proposed operators $\mathcal{RL}_{n}^{[\alpha]}$ defined in (\ref{o3}) for $\beta_n=n=5, 10, 20, 40$, respectively. The curves corresponding to these values are shown in magenta, red, green, and cyan, respectively. It is evident that the approximating curves approach the graph of $g(x)=\sin{x}$ as $n$ increases, confirming the convergence of the proposed operators. 

\begin{figure}[h!]
    \centering 
    \includegraphics[width=.52\textwidth]{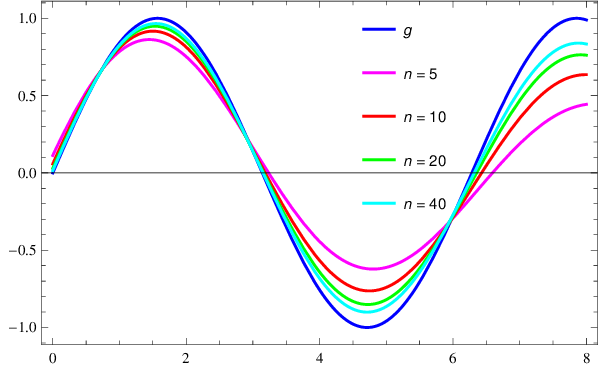}   
    \caption[Description in LOF, taken from~\cite{source}]{Graphical illustration of the convergence of the proposed operators  $\mathcal{RL}_{n}^{[\alpha]}(g;x)$  to the function $g(x)=\sin{x}$ for $\alpha=\frac{1}{50}$, $\phi=0.1$ and $~\psi=0.3$.}
    \label{F1}
\end{figure}
\end{example}

\begin{example}
Consider the function $g(x)=e^x$ and $\alpha=\frac{1}{50}$, $\phi=0.1,~\psi=0.3$. Figure \ref{F2} compares the proposed operators $\mathcal{RL}_{n}^{[\alpha]}(g;x)$ with the operators $\mathcal{L}_{n}^{[\alpha]}(g;x)$ defined in (\ref{o2}).
Throughout this example, we take $\beta_n=n$, where $n=1,2,3,4,5,6,7,8,9,10$. The graphs of $\mathcal{RL}_{n}^{[\alpha]}(g;x)$ and $\mathcal{L}_{n}^{[\alpha]}(g;x)$. Blue color and red color represent the operators $\mathcal{RL}_{n}^{[\alpha]}(g;x)$ and $\mathcal{L}_{n}^{[\alpha]}(g;x)$ are shown in blue and red, respectively.
\begin{figure}[h!]
    \centering 
    \includegraphics[width=.52\textwidth]{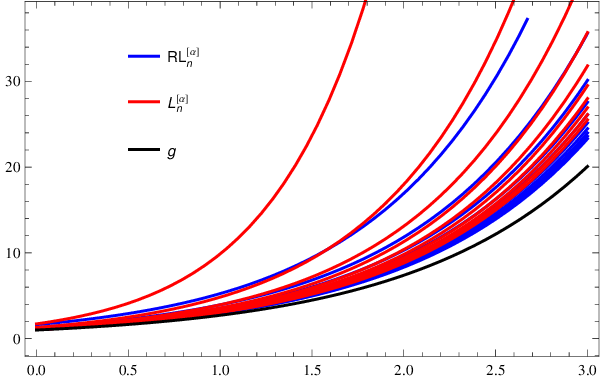}   
    \caption[Description in LOF, taken from~\cite{source}]{Comparison of convergence behaviour of $\mathcal{RL}_{n}^{[\alpha]}(g;x)$ and $\mathcal{L}_{n}^{[\alpha]}(g;x)$ towards $g(x)$.}
    \label{F2}
\end{figure}
\end{example}
Figure \ref{F2} indicates that, for the chosen parameters $\phi=0.1,~\psi=0.3$ the proposed operators $\mathcal{RL}_{n}^{[\alpha]}(g;x)$ provide a closer approximation to $g(x)=e^x$ than the operator $\mathcal{L}_{n}^{[\alpha]}(g;x)$ for the values of $n$ considered. In particular, the graphs of $\mathcal{RL}_{n}^{[\alpha]}(g;x)$ approach the graph of $g(x)$ more rapidly than those of $\mathcal{L}_{n}^{[\alpha]}(g;x)$. This numerical example illustrates the influence of the Stancu parameters on the approximation behaviour of the proposed operators.\\

To further illustrate the effect of the sequence $\beta_n$, we choose $\beta_n=n^2$ while keeping $\phi=0.1,~\psi=0.3$ and $\alpha=0.02$. Figure \ref{Fig2} compares the proposed operators $\mathcal{RL}_{n}^{[\alpha]}(g;x)$  with the operators $\mathcal{L}_{n}^{[\alpha]}(g;x)$ for the function $g(x)=e^x$. 

\begin{figure}[h!]
    \centering 
    \includegraphics[width=.52\textwidth]{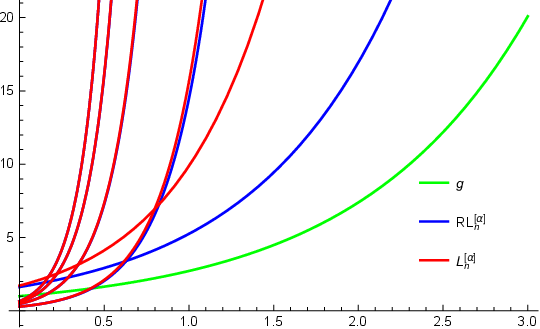}   
    \caption[Description in LOF, taken from~\cite{source}]{Comparison of  operators $\mathcal{RL}_{n}^{[\alpha]}(g;x)$ with operators $\mathcal{L}_{n}^{[\alpha]}(g;x)$.}
    \label{Fig2}
\end{figure}
The blue and red curves represent the operators $\mathcal{RL}_{n}^{[\alpha]}(g;x)$ and $\mathcal{L}_{n}^{[\alpha]}(g;x)$ respectively, while the black curve denotes the function $g(x)=e^x$. For the values of $n$ considered, the proposed operators $\mathcal{RL}_{n}^{[\alpha]}(g;x)$   remain closer to the graph of $g(x)$ than the operators $\mathcal{L}_{n}^{[\alpha]}(g;x)$ indicating an improved approximation for the selected parameters and sequence $\beta_n=n^2$. 
 

\begin{example}
Consider the function $g(x)=e^x$, as in the previous example. We compute the absolute approximation errors of the operators (\ref{o2}) and (\ref{o3}) at several points for different values of $\beta_n=n$ Throughout this example, we choose $\phi=0.1, \psi=0.9$. The corresponding numerical errors are reported in Tables \ref{t1} and \ref{t2}.

\end{example}

\begin{table}[ht]
\centering
\begin{tabular}{|c|c|c|c|c|c|c|}
\hline 
$x\downarrow$, $n\to$ & 5 & 10 & 20 & 30 & 40 & 50\\ 
\hline 
0.1 & 1.11666 & 1.06464 & 1.03632 & 1.02651 & 1.02153 & 1.01852\\ 
\hline 
0.5 & 1.26644 & 1.21685 & 1.18838 & 1.17828 & 1.17311 & 1.16997\\ 
\hline 
0.9 & 1.66465 & 1.63087 & 1.60865 & 1.60032 & 1.59597 & 1.5933\\ 
\hline 
1.0 & 1.81521 & 1.78871 & 1.7697 & 1.76235 & 1.75847& 1.75608 \\ 
\hline 
\end{tabular} 
\caption{Absolute approximation errors of $\mathcal{RL}_{n}^{[\alpha]}(g;x)$ in approximating $g(x)=e^x$
}\label{t1}
\end{table}
\begin{table}[ht]
\centering
\begin{tabular}{|c|c|c|c|c|c|c|}
\hline 
$x\downarrow$, $n\to$ & 5 & 10 & 20 & 30 & 40 & 50\\
\hline 
0.1 & 1.13814 & 1.06966 & 1.03736 & 1.02687 & 1.02168 & 1.01858 \\ 
\hline 
0.5 & 1.43749 & 1.2894 & 1.22141 & 1.19958 & 1.18881 & 1.1824 \\ 
\hline 
0.9 & 2.13185 & 1.83752 & 1.70538 & 1.66335 & 1.6427 & 1.63042 \\ 
\hline 
1.0 & 2.39098 & 2.04452 & 1.8898 & 1.84069 & 1.81658 & 1.80226 \\ 
\hline 
\end{tabular} 
\caption{Absolute approximation errors of $\mathcal{L}_{n}^{[\alpha]}(g;x)$ in approximating function $g(x)=e^x$}\label{t2}
\end{table}
\textbf{Concluding Remark:} Tables \ref{t1} and \ref{t2} show that the approximation errors decrease as $n$ increases, indicating the convergence of both sequences of operators. Furthermore, for all values of $x$ and $n$ considered in the tables, the proposed operators $\mathcal{RL}_{n}^{[\alpha]}(g;x)$ produce smaller approximation errors than the operators $\mathcal{L}_{n}^{[\alpha]}(g;x)$. These numerical results demonstrate that, for the chosen values of the Stancu parameters $\phi$ and $\psi$ the proposed operators provide a more accurate approximation of $g(x)=e^x$. The numerical experiments suggest that an appropriate choice of the parameters $\phi$ and $\psi$ improves the approximation performance of the proposed operators. 

\subsection{\textbf{Effect of $\alpha$ on the approximation}}
In this subsection, we investigate the influence of the parameter $\alpha$ on the approximation behaviour of the operators (\ref{o2}) and (\ref{o3}). A numerical comparison is presented in Tables \ref{t3} and \ref{t4}. Throughout this example, we take $x=0.1$, $\phi=0.1$ and $\psi=0.9$.




\begin{table}[ht]
\centering
\begin{tabular}{|c|c|c|c|c|c|c|c|c|c|c|}
\hline 
$n\downarrow$, $\alpha\to$ & $\frac{1}{5}$ & $\frac{1}{10}$ & $\frac{1}{20}$ & $\frac{1}{30}$ & $\frac{1}{50}$ & $\frac{1}{100}$ & $\frac{1}{150}$ & $\frac{1}{200}$ & $\frac{1}{250}$ & $\frac{1}{500}$\\ 
\hline 
5 & 1.12752 & 1.12115 & 1.11828 & 1.11737 & 1.11666 & 1.11613 & 1.11595 & 1.11587 & 1.11581 & 1.11571 \\ 
\hline 
10 & - & 1.06931 & 1.06633 & 1.06539 & 1.06464 & 1.0641 & 1.06391 & 1.06382 & 1.06377 & 1.06366 \\ 
\hline 
15 & - & - & 1.04765 & 1.0467 & 1.04595 & 1.04539 & 1.04521 & 1.04512 & 1.04506 & 1.04495 \\ 
\hline 
20 & - & - & 1.03804 & 1.03708 & 1.03632 & 1.03577 & 1.03558 & 1.03549 & 1.03544 & 1.03533 \\ 
\hline 
\end{tabular} 
\caption{Effect of $\alpha$ on the approximation by the operators $\mathcal{RL}_{n}^{[\alpha]}(g;x)$}\label{t3}
\end{table}
\begin{table}[ht]
\centering
\begin{tabular}{|c|c|c|c|c|c|c|c|c|c|c|}
\hline 
$n\downarrow$, $\alpha\to$ & $\frac{1}{5}$ & $\frac{1}{10}$ & $\frac{1}{20}$ & $\frac{1}{30}$ & $\frac{1}{50}$ & $\frac{1}{100}$ & $\frac{1}{150}$ & $\frac{1}{200}$ & $\frac{1}{250}$ & $\frac{1}{500}$\\  
\hline 
5 & 1.15457 & 1.14482 & 1.14054 & 1.13919 & 1.13814 & 1.13737 & 1.13711 & 1.13698 & 1.13691 & 1.13675 \\ 
\hline 
10 & - & 1.07532 & 1.0717 & 1.07055 & 1.06966 & 1.069 & 1.06878 & 1.06867 & 1.0686 & 1.06847 \\ 
\hline 
15 & - & - & 1.04992 & 1.04884 & 1.04799 & 1.04736 & 1.04716 & 1.04705 & 1.04699 & 1.04687 \\ 
\hline 
20 & - & - & 1.04687 & 1.03819 & 1.03736 & 1.03675 & 1.03655 & 1.03645 & 1.03639 & 1.03627 \\ 
\hline 
\end{tabular} 
\caption{Effect of $\alpha$ on the approximation by  the operators $\mathcal{L}_{n}^{[\alpha]}(g;x)$}\label{t4}
\end{table}
To further investigate the influence of $\alpha$, we choose $\beta_n=n$ and consider several admissible choices of $\alpha$, namely
\begin{eqnarray*}
    \alpha=\frac{1}{n},~\frac{1}{n+1},~\frac{1}{(n^2-n+1)},~\frac{1}{n^2},~\frac{1}{(n^2+n+1)},~\frac{1}{(n+1)^2}.
\end{eqnarray*}
Table \ref{t5} reports the corresponding absolute approximation errors of the proposed operators (\ref{o3}) for the function $g(x)e^x$ at the point $x=0.5$. 
\begin{table}[ht]
\centering
\begin{tabular}{|c|c|c|c|c|c|c|}
\hline 
$n\downarrow$, $\alpha\to$  & $\frac{1}{n}$ & $\frac{1}{n+1}$ & $\frac{1}{n^2-n+1}$ &  $\frac{1}{n^2+\frac{1}{2}}$ & $\frac{1}{n^2+n+1}$ & $\frac{1}{(n+1)^2}$  \\ 
\hline 
15 & 0.488391 & 0.483345 & 0.418235 &  0.417913 & 0.417613 & 0.417357  \\ 
\hline 
25 & 0.422315 & 0.420564 & 0.380506 &  0.380438 & 0.380372 & 0.380313  \\ 
\hline 
30 & 0.405894 & 0.404689 & 0.371123 &  0.371084 & 0.371046 & 0.371012  \\ 
\hline 
50 & 0.373167 & 0.372742 & 0.352398 &  0.352389 & 0.352381 & 0.352373  \\ 
\hline 
100 & 0.348721 & 0.348617 & 0.338375 & 0.338374 & 0.338373 & 0.338372  \\ 
\hline 
200 &  0.33653 & 0.336505 & 0.331368 &  0.331368 & 0.331367 & 0.331367  \\ 
\hline 
\end{tabular} 
\caption{
Absolute approximation errors of the proposed operators $\mathcal{RL}_{n}^{[\alpha]}(g;x)$ for different choices of the parameter $\alpha$}\label{t5}
\end{table}

Tables \ref{t3}--\ref{t5} illustrate the impact of the parameter $\alpha$ on the approximation behaviour of the proposed operators. The numerical results indicate that,  the approximation error decreases as smaller values of $\alpha$ are chosen. Moreover, the proposed Stancu-type operators consistently produce smaller numerical errors than the corresponding operators (\ref{o2}) for the selected values of the parameters. These observations suggest that an appropriate choice of $\alpha$, together with the Stancu parameters $\phi$ and $\psi$, can improve the approximation performance of the proposed operators.

\section{Conclusion and result Discussion}\label{sec9}

In this paper, we introduced a modified Stancu variant of the Sz$\acute{\text{a}}$sz-Mirakjan-Kantorovich operators involving the additional parameters $\phi$, $\psi$, and the sequence ${\beta_n}$. We investigated their approximation behaviour by establishing local and weighted approximation results, estimates in terms of the modulus of continuity and the second-order modulus of smoothness, Voronovskaya-type, quantitative Voronovskaya-type, and Grüss-Voronovskaya-type theorems, together with convergence results for functions whose derivatives are of bounded variation. Numerical examples and graphical illustrations were presented to support the theoretical analysis and to demonstrate the influence of the parameters on the approximation process.

The proposed family contains several well-known operators as special cases, providing a unified framework for studying their approximation properties. The additional parameters offer greater flexibility in constructing approximation processes and may be useful in developing further generalizations of positive linear operators. Possible directions for future work include extending these operators to multivariate settings, (q)-analogues, statistical approximation, and approximation in more general function spaces.

\end{document}